\documentclass[11pt]{article}
\usepackage[ total={6.5in, 8in}]{geometry}
\usepackage{enumerate}
\usepackage{amsmath}
\usepackage{amssymb,latexsym}
\usepackage{amsthm}
\usepackage{color}
\usepackage{cancel}
\usepackage{graphicx}
\usepackage{cite}
\usepackage{changes}

\makeatletter
\usepackage{comment}
\let\wfs@comment@comment\comment
\let\comment\@undefined

\usepackage{changes}
\let\wfs@changes@comment\comment
\let\comment\@undefined

\newcommand\comment{%
    \ifthenelse{\equal{\@currenvir}{comment}}
    {\wfs@comment@comment}
    {\wfs@changes@comment}%
}
\definechangesauthor[name=Daniele, color=blue]{DAN}
\definechangesauthor[name=Marco, color=red]{MARCO}

\usepackage{hyperref}
\usepackage{comment}
\usepackage{mathrsfs}
\usepackage{amscd}

\newtheorem*{MainTheorem}{Main Theorem}
\newtheorem{theorem}{Theorem}[section]

\newtheorem{lemma}[theorem]{Lemma}

\newtheorem{proposition}[theorem]{Proposition}

\newcommand{\cH}{{\mathcal H}}

\newcommand{\cQ}{{\mathcal Q}}
\newcommand{\cX}{{\mathcal X}}

\newcommand{\fq}{{\mathbb F}_{q}}

\newcommand{\Fqs}{\mathbb{F}_{q^6}}
\newcommand{\PGqs}{\mathrm{PG}(2,\mathbb{F}_{q^6})}
\newcommand{\AGqs}{\mathrm{AG}(2,\mathbb{F}_{q^6})}

\title{Complete $(q+1)$-arcs in $\mathrm{PG}(2,\mathbb{F}_{q^6})$ from the Hermitian curve}
\author{Daniele Bartoli\thanks{Dipartimento di Matematica e Informatica, Universit\`a degli Studi di Perugia,  Perugia, Italy. daniele.bartoli@unipg.it} and 
Marco Timpanella\thanks{Dipartimento di Matematica e Informatica, Universit\`a degli Studi di Perugia,  Perugia, Italy.
marco.timpanella@unipg.it}}
\date{}

\begin{document}
\maketitle
\begin{abstract}
We prove that, if $q$ is large enough, the set of the $\mathbb{F}_{q^6}$-rational points of the Hermitian curve is a complete $(q+1)$-arc in $\PGqs$, addressing an open case from a recent paper by Korchm\'aros, Szőnyi and Nagy. An algebraic approach based on the investigation of some algebraic varieties attached to the arc is used.
\end{abstract}

\section{Introduction}


Let $\mathrm{PG}(2,\mathbb{F}_q)$ be the projective Galois plane over the field $\mathbb{F}_q$ with $q$ elements, where $q=p^h$ with $p$ be a prime and $h$ a positive integer.
A $(k, m)$-arc $\mathcal{A}$ in $\mathrm{PG}(2,\mathbb{F}_q)$ is a point set of size $k$, containing $m$ collinear points but such that no $m + 1$ points lie on the same line.  The arc $\mathcal{A}$ is said \emph{complete} if it is maximal with respect to set theoretical inclusion. 
Arcs in projective planes have a useful interpretation in Coding theory as the matrix $M$ whose columns are homogeneous coordinates of the points of a $(k,m)$-arc of $\mathrm{PG}(2,\mathbb{F}_q)$ with $m\geq 3$ is the parity check matrix of a $[k,k-3,3]_q$ almost MDS code which is non-extendible if and only if the corresponding $(k,m)$-arc is complete. Also, for fixed $m$ and $q$, the larger is $k$, the closer is the $[k,3,k-m]_q$ code generated by $M$ to the Griesmer bound.
 
The foundation of the theory of $(k,m)$-arcs  dates back to the seminal works of Segre and Barlotti in the 1950s. A general introduction to $(k, m)$-arcs can be found in the monograph \cite[Chapter 12]{MR1612570}, as well as in the survey paper \cite[Section 5]{MR2061806}. 

A natural example of a $(k,m)$-arc is the set $\mathcal{A}=\cX(\mathbb{F}_q)$ of $\mathbb{F}_q$-rational points of a plane curve $\cX$ (defined over $\mathbb{F}_q$) in $\mathrm{PG}(2,\overline{\mathbb{F}}_q)$ having no linear components; here $\overline{\mathbb{F}}_q$ denotes the algebraic closure of $\mathbb{F}_q$, $k=\#\cX(\mathbb{F}_q)$, and $m$ is less than or equal to the degree of $\cX$. However, establishing whether $\mathcal{A}$ is complete can be a very challenging problem which has received a lot of attention in the last few years. In particular, the case where there are only few possibilities for the characters of $\mathcal X(\fq)$ (i.e. the sizes of the intersection of $\mathcal X(\fq)$ and a line $\ell$ in $\textrm{PG}(2,\mathbb F_q)$) appears to be hard. This corresponds to the case where the code generated by the matrix
whose columns are homogeneous coordinates of the points of $\mathcal X({\fq})$
has few characters.

For $m\ge 3$, the best known curve with few characters in $\textrm{PG}(2,\mathbb F_{\bar{q}})$ is the Hermitian curve $\mathcal H_{ q}$ with affine equation 
$$
X^{ q+1}=Y^{q}+Y
$$
where $\bar q=q^{2r}$. The characters of $\mathcal H_{{ q}}$ in fact are just $0,1,2,{q}+1$. Moreover, when $r$ is odd, $\mathcal H_{{ q}}$ has the maximum number of points in $\mathrm{PG}(2,\mathbb F_{\bar{q}})$ for a degree $q+1$ curve, and hence $\mathcal H_{{ q}}(\mathbb F_{\bar{q}}) $ is the largest $(k,{ q}+1)$-arc that arises from an algebraic curve of degree ${ q}+1$, see \cite{HKT} for more details on maximal curves.

The completeness of 
$\mathcal H_{{ q}}$ in 
$\textrm{PG}(2,\mathbb{F}_{ q^{2r}})$ 
has been recently investigated in \cite{K-Hermitian}.
Using an approach based on Galois Theory together with $\rm{\check{C}}$ebotarev type density theorems, previously proposed in  \cite{MR4520241},  the authors proved that that for $r\geq 5$ the $\mathbb{F}_{{q}^{2r}}$-rational points of $\cH_q$ form a complete $({ q}+1)$-arc.

In this paper we focus on the remaining case $r=3$, using quite different techniques with respect to \cite{K-Hermitian}. In fact, the above mentioned Galois theory approach does not allow to prove the case $r=3$ due to the need of asymptotic results such as $\rm{\check{C}}$ebotarev type density theorem. The main idea in our approach is to translate the collinearity condition into the existence of suitable points in algebraic varieties of dimension $3$ over $\mathbb{F}_q$. On the one hand this makes the computations more involved; on the other hand the degree of these algebraic varieties is now fixed and small with respect to $q$, which allows to exploit asymptotic results such as Theorem \ref{Th:CafureMatera}. 

Our main result is the following.

\begin{MainTheorem}\label{Main theorem}
Let $\cH_q$ be the Hermitian curve of affine equation $X^{q+1}=Y^q+Y$. If $q$ is large enough, the set of the $\mathbb{F}_{q^6}$-rational points of $\cH_q$ is a complete $(|\cH(\mathbb{F}_{q^{6}})|,q+1)$-arc in $\mathrm{PG}(2,\mathbb{F}_{q^6})$.
\end{MainTheorem}

\section{Preliminaries on algebraic varieties}
In this section, we will provide some of the tools that will be used throughout the paper. For a more comprehensive introduction on algebraic varieties we refer to \cite{HKT}.
We first introduce the notations that will be used throughout the paper.
Let $p$ be a prime, $h$ be a positive integer, and $q=p^h$. Denote by $\mathrm{PG}(2,\mathbb{F}_q)$ the projective Galois plane over the field $\mathbb{F}_q$ with $q$ elements. The symbol $\overline{\mathbb{F}}_q$ denotes the algebraic closure of $\mathbb{F}_q$. The corresponding affine plane of order $q$ is denoted by $\mathrm{AG}(2,\mathbb{F}_q)=\mathrm{PG}(2,\mathbb{F}_q)\setminus \ell_{\infty}$, where $\ell_{\infty}$ is the line at infinity.  

A variety and more specifically a curve, i.e. a variety of dimension 1, are described by a certain set of equations with coefficients  in a finite field $\mathbb{F}_q$. A variety  defined by a unique equation is called a hypersurface. We say that a variety $\mathcal{V}$ is \emph{absolutely irreducible} if there are no varieties $\mathcal{V}^{\prime}$ and $\mathcal{V}^{\prime\prime}$ defined over the algebraic closure of $\mathbb{F}_q$ and different from $\mathcal{V}$ such that $\mathcal{V}= \mathcal{V}^{\prime} \cup \mathcal{V}^{\prime\prime}$. If a variety $\mathcal{V}\subset \mathrm{PG}(r,\mathbb{F}_q)$ is defined by $F_i(X_0,\ldots, X_r)=0$, for  $i=1,\ldots, s$, an $\mathbb{F}_{q}$-rational point of $\mathcal{V}$ is a point $(x_0:\ldots:x_r) \in \mathrm{PG}(r,\mathbb{F}_q)$ such that $F_i(x_0,\ldots, x_r)=0$, for  $i=1,\ldots s$. A point is affine if $x_0\neq 0$. The set of the $\mathbb{F}_q$-rational points of $\mathcal{V}$ is usually denoted by $\mathcal{V}(\mathbb{F}_q)$. We usually denote by the same symbol homogenized polynomials and their dehomogenizations, if the context is clear.

In our machinery it will be crucial to estimate the number of $\mathbb{F}_q$-rational points of certain algebraic varieties.  The first estimate on the number of $\mathbb{F}_{q}$-rational points of an algebraic variety was given by Lang and Weil \cite{MR65218} in 1954. In this paper we will use the following refinement due to Cafure and Matera \cite{MR2206396}, where the estimate is given in terms of the dimension and the degree of the variety. We recall that the degree of an algebraic variety of dimension $r$ in $\mathbb{P}^n(\mathbb{F}_q$) is defined as the 
number of intersection points  of the variety with a general linear subspace of dimension $n-r$.

\begin{theorem}\cite[Theorem 7.1]{MR2206396}\label{Th:CafureMatera}
Let $\mathcal{V}\subset\mathrm{AG}(N,\mathbb{F}_q)$ be an absolutely irreducible variety defined over $\mathbb{F}_q$ of dimension $r>0$ and degree $\delta$. If $q>2(r+1)\delta^2$, then the following estimate holds:
$$|\mathcal{V}(\mathbb{F}_{q})\cap \mathrm{AG}(N,\mathbb{F}_q)-q^r|\leq (\delta-1)(\delta-2)q^{r-1/2}+5\delta^{13/3} q^{r-1}.$$
\end{theorem}

\section{Construction of a $(q+1)$-arc in $\mathrm{PG}(2,\mathbb{F}_{q^6})$ from the Hermitian curve} 

Henceforth, let $\mathcal{H}_q$ denote the Hermitian curve with the affine equation
$$
X^{q+1}=Y^q+Y.
$$
In order to prove our Main Theorem, in this section we will show the existence of an $\mathbb{F}_{q^6}$-rational $(q+1)$-secant to $\cH_q(\mathbb{F}_{q^6})$ through any point in $\PGqs$, i.e. an  $\mathbb{F}_{q^6}$-rational line meeting $\cH_q(\mathbb{F}_{q^6})$ in exactly $q+1$ distinct points. As pointed out in \cite{K-Hermitian}, it is sufficient to prove our claim for points in $\AGqs$, thanks to the action of $\mathrm{PGU}(3,\mathbb{F}_{q})$ leaving invariant $\cH_q(\Fqs)$ and preserving no lines in $\PGqs$.

Note that, as $\mathcal{H}_q$ is an $\Fqs$-maximal curve, we have $|\cH_q(\mathbb{F}_{q^6})|=q^6+1+q(q-1)q^3$. 


From now on, let  $P=(a,b)$ be a point in $\AGqs$, and 
 $\ell$ be an $\mathbb{F}_{q^6}$-rational line through $P$ with affine equation
$$
Y=m(X-a)+b,
$$
where $m\in \Fqs$.  Then, the  size of $|\cH_q\cap\ell|$ in $\PGqs$ depends on the number of roots in $\Fqs$ of the polynomial
\begin{equation*}\label{eq:polIntersez}
f(X):=X^{q+1}-(m(X-a)+b)^{q}-(m(X-a)+b)=X^{q+1}-m^qX^q-mX-ma-m^qa^q-b-b^q.
\end{equation*}

For the following results on the roots of $f(X)$ and their connections with the intersections of $\cH_q$ and $\ell$ in $\PGqs$, we refer to \cite[Section 4.1]{K-Hermitian}.

\begin{proposition}\label{prop1}
Let $a,b,m\in \mathbb{F}_{q^6}$. Then, the polynomial $f(X)$ has either $0$, $1$, $2$ or $q+1$ roots in $\overline{\mathbb{F}}_q$. In particular, $f(X)$ has multiple roots if and only if $m^{q+1}+m^qa^q+ma+b^q+b=0$, and in this case $f(X)=(X-m^q)(X^q-m)$.
\end{proposition}

\begin{proposition}\label{propK1}
Let $a,b,m\in \mathbb{F}_{q^6}$. Then $\ell$ is a $(q+1)$-secant to $\cH_q(\mathbb{F}_{q^6})$ if and only if $f(X)$ has $q+1$ distinct roots  in $\mathbb{F}_{q^6}$.
\end{proposition}

By Proposition \ref{propK1}, to prove Theorem \ref{Main theorem} we must show that for each $a,b\in \Fqs$ there exists $m\in\Fqs$ such that $f(X)$ has $q+1$ distinct roots in $\Fqs$.

Our approach is the following. Henceforth, we let 
$$\gamma:=b+b^q$$ 
and  
$$
A:=(a^{q+1}-a^{q^2+q})+(a^{q+1}-a^{q^2+q})^{q^2}+(a^{q+1}-a^{q^2+q})^{q^4}=\text{Tr}_{q^2}^{q^6}(a^{q+1}-a^{q^2+q}).
$$
Clearly $\gamma-\gamma^q+\gamma^{q^2}-\gamma^{q^3}+\gamma^{q^4}-\gamma^{q^5}=0$ and $f(X)=X^{q+1}-m^qX^q-mX-ma-m^qa^q-\gamma.$ Also, observe that $a\in\mathbb{F}_{q^2}$ yields $A=0$. 
Now, let $x\in\mathbb{F}_{q^6}$ be a root of $f(X)$. By Proposition \ref{prop1}, if $m^{q+1}+m^qa^q+ma+b^q+b\neq 0$, then $x\neq m^q$.
Therefore, we can write
\begin{eqnarray*}
    x^q&=&\frac{mx+ma+m^qa^q+\gamma}{x-m^q}:=x_q,\\
    x^{q^2}&=&\frac{m^qx_q+m^qa^q+m^{q^2}a^{q^2}+\gamma^q}{x_q-m^{q^2}}:=x_{q^2},\\
    x^{q^3}&=&\frac{m^{q^2}x_{q^2}+m^{q^2}a^{q^2}+m^{q^3}a^{q^3}+\gamma^{q^2}}{x_{q^2}-m^{q^3}}:=x_{q^3},\\
    x^{q^4}&=&\frac{m^{q^3}x_{q^3}+m^{q^3}a^{q^3}+m^{q^4}a^{q^4}+\gamma^{q^3}}{x_{q^3}-m^{q^4}}:=x_{q^4},\\
    x^{q^5}&=&\frac{m^{q^4}x_{q^4}+m^{q^4}a^{q^4}+m^{q^5}a^{q^5}+\gamma^{q^4}}{x_{q^4}-m^{q^5}}:=x_{q^5},\\
    x^{q^6}&=&\frac{m^{q^5}x_{q^5}+m^{q^5}a^{q^5}+ma+\gamma^{q^5}}{x_{q^5}-m}:=x_{q^6}.\\
\end{eqnarray*}

As $x=x^{q^6}$, $x$ must be a root of $h(X):=X_{q^6}-X$. By a MAGMA computation, $h(X)=0$ reads
\begin{equation}\label{x6-X0}
\frac{f_3(m,m^{q},\ldots,m^{q^5})+f_2(m,m^{q},\ldots,m^{q^5})X+f_1(m,m^{q},\ldots,m^{q^5})X^2}{f_5(m,m^{q},\ldots,m^{q^5})+f_4(m,m^{q},\ldots,m^{q^5})X}=0,
\end{equation}
where the expressions of $f_i(\mathbf{y})$, $i=1,\ldots,5$ are given in Section \ref{Appendice}.

In terms of the polynomials $f_i(X)$, the following holds.
\begin{proposition}\label{prop2}
Let $a,b\in \mathbb{F}_{q^6}$ and $m\in \mathbb{F}_{q^6}$ such that $m^{q+1}+m^qa^q+ma+b^q+b\neq 0$. Then, the line $Y=m(X-a)+b$ intersects $\cH_q(\Fqs)$ in exactly $q+1$ distinct points in $\PGqs$ if and only if $f_i(m,m^{q},\ldots,m^{q^5})=0$, $i=1,2,3$ and $(f_4(m,m^{q},\ldots,m^{q^5}), f_5(m,m^{q},\ldots,m^{q^5}))\neq (0,0)$. 
\end{proposition}
\begin{proof}
As $m^{q+1}+m^qa^q+ma+b^q+b\neq 0$, $f(X)$ has $q+1$ distinct roots $\{\overline{x}_0,\ldots,\overline{x}_q\}$ in $\overline{\mathbb{F}}_q$. By the definition of $h(X)$, $\overline{x}_i\not\in \mathbb{F}_{q^6}$ if and only if either $f_5(m,m^{q},\ldots,m^{q^5})+f_4(m,m^{q},\ldots,m^{q^5})\overline{x}_i=0$ or $h(x_i)\neq 0$. Assume first that all the roots of $f(X)$ lie in $\mathbb{F}_{q^6}$. Then $h(\overline{x}_i)=0$ and $f_5(m,m^{q},\ldots,m^{q^5})+f_4(m,m^{q},\ldots,m^{q^5})\overline{x}_i\neq 0$ for any $i\in \{1,\ldots,q\}$. Therefore $f_3(m,m^{q},\ldots,m^{q^5})+f_2(m,m^{q},\ldots,m^{q^5})X+f_1(m,m^{q},\ldots,m^{q^5})X^2$ is the zero polynomial whereas $f_5(m,m^{q},\ldots,m^{q^5})+f_4(m,m^{q},\ldots,m^{q^5})X$ is not the zero polynomial.

On the other hand, if $f_i(m,m^{q},\ldots,m^{q^5})=0$, $i=1,2,3$ and $$(f_4(m,m^{q},\ldots,m^{q^5}), f_5(m,m^{q},\ldots,m^{q^5})\neq (0,0),$$ then at least $q$ roots of $f(X)$ lie in $\mathbb{F}_{q^6}$. So, all of them lie in $\mathbb{F}_{q^6}$. 
The claim now follows because of the correspondence between the intersections of $\cH_q$ and $\ell$ in $\PGqs$ and the roots of $f(X)$ in $\Fqs$.
\end{proof}

We will now prove Theorem \ref{Main theorem} in two steps, according to $A\neq 0$ or $A=0$. Actually, our computations provide additional information. Indeed, we prove that through a point such that $A\neq 0$, the number of $(q+1)-$secants to $\cH_q(\Fqs)$ is in the order of $q^3$, see Theorem \ref{mainsec1}. On the other hand, for a point such that $A=0$, we show that there always exists a $(q+1)$-secant with slope in $\mathbb{F}_{q^2}$, see Theorem \ref{mainsec2}. 

\subsection{The case $A\neq 0$}

The aim of this section is to prove the following result.

\begin{theorem}\label{mainsec1}
Through any point $(a,b)\in \AGqs$ such that $A\neq 0$ there are  $O(q^3)$ $\Fqs$-rational $(q+1)$-secants to $\mathcal{H}_q(\Fqs)$.
\end{theorem}

To prove Theorem \ref{mainsec1}, we first prove the following preliminary results.

\begin{proposition}\label{Propmain1}
 Let $(a,b)\in \AGqs$ such that $A\neq 0$. Then, there are at least $$q^3-3422q^{5/2}-5\cdot 60^{13/3}q^2-9q^2$$ values $m\in \mathbb{F}_{q^6}$ such that $$f_1(m,m^{q},\ldots,m^{q^5})=f_2(m,m^{q},\ldots,m^{q^5})=f_3(m,m^{q},\ldots,m^{q^5})=0.$$
\end{proposition}
\begin{proof}
Consider the polynomials $g_1({\bf y}), g_2({\bf y}), g_3({\bf y})$ as in Section \ref{Appendice}.

First, we will prove that there are $O(q^3)$ values of $m\in \mathbb{F}_{q^6}$ that satisfy the three equations
\begin{equation}\label{Eq:g}g_1(m,m^q,\ldots,m^{q^5})=g_2(m,m^q,\ldots,m^{q^5})=g_3(m,m^q,\ldots,m^{q^5})=0.\end{equation}

Indeed, let $\mathcal{W}$ be the variety of $\mathrm{PG}(6,\Fqs)$ defined by the affine equations
\begin{equation*}
    \begin{cases}
    g_1(x_0,\ldots,x_5)=0\\
    g_2(x_0,\ldots,x_5)=0\\
    g_3(x_0,\ldots,x_5)=0.\\
    \end{cases}
\end{equation*}

This variety contains an absolutely irreducible component of dimension $3$. This follows from the fact that $g_1$, $g_2$, $g_3$ are of degree one in $x_1$, $x_3$, $x_0$ respectively.
 Indeed, let $G_1$ be the coefficient of $x_1$ in $g_1$, $G_2$ be the coefficient of $x_3$ in $g_2$, and $G_3$ be the coefficient of $x_0$ in $g_3$. Then, $G_1,G_2,G_3\in \mathbb{F}_{q^6}[x_0,x_1,x_2,x_3,x_4,x_5]$ are not the zero polynomial. To see this, observe that the coefficient of $x_2x_4x_5^2$ in $G_3$ and of $x_2x_4x_5^2$ in $G_2$ is $A$, which is non-zero by assumption; on the other hand, the coefficient of $x_2x_4$ in $G_1$ is $a^{q^2}-a^{q^4}$, and hence it is non-zero since $a^{q^4}-a^{q^2}=0$ would imply $A=0$, a contradiction. Also, $G_1,G_2,G_3$ have degrees at most $2$, $3$, and $4$, respectively, whence $G_1G_2G_3$ is a polynomial of degree at most $9$ depending only on $x_2,x_4,x_5$. This means that there are at most $9q^2$ triples $(\overline{x}_2,\overline{x}_4,\overline{x}_5)\in \mathbb{F}_{q^6}^3$ such that $$G_1(\overline{x}_2,\overline{x}_4,\overline{x}_5)G_2(\overline{x}_2,\overline{x}_4,\overline{x}_5)G_3(\overline{x}_2,\overline{x}_4,\overline{x}_5)=0.$$

Now, let $\{\xi,\ldots,\xi^{q^5}\}$ be a normal basis of $\mathbb{F}_{q^6}$ over $\mathbb{F}_q$, and $\eta$ be the change of variables given by
$$y_0=\sum_{i=0}^{5} \xi^{q^i}x_i, \quad  y_1=\sum_{i=0}^{5} \xi^{q^{i+1}}x_i, \quad \ldots, \quad y_5=\sum_{i=0}^{5} \xi^{q^{i+5}}x_i.$$
Also, denote $\overline{g}_i(x_0,\ldots,x_5)=g_i\circ \eta$. Then, the ideal  $\mathcal{I}=\langle \overline{g}_1,\overline{g}_2,\overline{g}_3\rangle$ is fixed by the Frobenius automorphism
$$ \sum_{(i_0,\ldots,i_5)} \alpha_{(i_0,\ldots,i_5)}x_0^{i_0}x_1^{i_1}x_2^{i_2}x_3^{i_3}x_4^{i_4}x_5^{i_5}\mapsto  \sum_{(i_0,\ldots,i_5)} \alpha_{(i_0,\ldots,i_5)}^q x_0^{i_0}x_1^{i_1}x_2^{i_2}x_3^{i_3}x_4^{i_4}x_5^{i_5},$$
and hence it is an $\mathbb{F}_q$-ideal of $\mathbb{F}_{q^6}[x_0,\ldots,x_5]$. This fact has been directly checked with MAGMA. Therefore, the variety $\mathcal{V}$ of $\mathrm{PG}(6,\Fqs)$ defined by the affine equations
\begin{equation*}
    \begin{cases}
    \overline{g}_1(x_0,\ldots,x_5)=0\\
    \overline{g}_2(x_0,\ldots,x_5)=0\\
    \overline{g}_3(x_0,\ldots,x_5)=0\\
    \end{cases}
\end{equation*}
is $\mathbb{F}_{q}$-rational and it is projectively equivalent to $\mathcal{W}$. 
 
Therefore, $\mathcal{V}$ contains an $\mathbb{F}_{q}$-rational absolutely irreducible algebraic variety of dimension 3 whose degree is at most $\deg(\overline{g}_1)\cdot\deg(\overline{g}_2)\cdot\deg(\overline{g}_3)=60$, and by Theorem \ref{Th:CafureMatera} there are at least
$$
q^3-59\cdot 58 q^{5/2}-5\cdot 60^{13/3}q^2-9q^2
$$
tuples $(\overline{x}_0,\ldots,\overline{x}_5)\in \mathbb{F}_{q}^6$  such that $\overline{g_{i}}(\overline{x}_0,\ldots,\overline{x}_5)=0$, $i=1,2,3$ and $$G_1(\overline{x}_0,\ldots,\overline{x}_5)G_2(\overline{x}_0,\ldots,\overline{x}_5)G_3(\overline{x}_0,\ldots,\overline{x}_5)\neq 0.$$
 By definition of the polynomials $\overline{g_{i}}$, all such tuples correspond to tuples $(\overline{m},\ldots,\overline{m}^{q^5})\in \mathbb{F}_{q^6}^{6}$  satisfying \eqref{Eq:g}. To conclude the proof, it is sufficient to check that such tuples $\overline{m}$  also satisfy $f_1(\overline{m})=f_2(\overline{m})=f_3(\overline{m})=0$. 
 \begin{comment}
 ZZZ := SequenceToSet(Coefficients(Numerator(x6-X0),X0));
ZZZ := {Resultant(pol,rrr,m1) : pol in ZZZ};
ZZZ := {Resultant(pol,rrr2,m3) : pol in ZZZ};
ZZZ := {Resultant(pol, B0-B1+B2-B3+B4-B5,B5): pol in ZZZ};
{IsDivisibleBy(pol,rrr3) : pol in ZZZ};
\end{comment}
\end{proof}

\begin{proposition}\label{Propmain2}
Let $(a,b)\in \AGqs$ such that $A\neq 0$. Then, there are at most $$q^2+1499\cdot 1500q^{1/2}+5\cdot 1500^{13/3}q$$ values $m\in \mathbb{F}_{q^6}$ such that $f_i(m,m^{q},\ldots,m^{q^5})=0$ for each $i\in\{1,\ldots,5\}$.
\end{proposition}
\begin{proof}
    We argue as in the proof of Proposition \ref{Propmain1}, and we adopt the same notation. So, let $\overline{f}_i(x_0,\ldots,x_5)=f_i\circ \eta$, $i=4,5$. We claim that the variety $\mathcal{W}$
\begin{equation*}
    \begin{cases}
    \overline{g}_1(x_0,\ldots,x_5)=0\\
    \overline{g}_2(x_0,\ldots,x_5)=0\\
    \overline{g}_3(x_0,\ldots,x_5)=0\\
    \overline{f}_4(x_0,\ldots,x_5)=0\\
    \overline{f}_5(x_0,\ldots,x_5)=0\\
    \end{cases}
\end{equation*}
is properly contained in $\mathcal{V}$. To do this, for $i=4,5$, let
\begin{eqnarray*}
    r_{i,1}&:=&\mathrm{Res}(\overline{f}_i,\overline{g}_1,x_1),\\
    r_{i,2}&:=&\mathrm{Res}(r_{i,1},\overline{g}_2,x_3),\\
    r_{i,3}&:=&\mathrm{Res}(r_{i,2},\overline{g}_3,x_0).
\end{eqnarray*}
Then, a MAGMA computation shows that $r_{4,3}$ is the zero polynomial, whereas  $r_{5,3}$ is not, and it factorizes as 
\begin{eqnarray*}
&&(x_5a^q -x_5a^{q^5} + \gamma^q - \gamma^{q^2} + \gamma^{q^3} - \gamma^{q^4})\cdot(x_4x_5 + x_4a^{q^4} + x_5a^{q^5} + \gamma^{q^4})\cdot\\&& (x_2x_4a^{q^2} - x_2x_4a^{q^4} + x_2a^{q^2+q}-x_2\gamma^q + x_2\gamma^{q^2} - x_2\gamma^{q^3} - x_4a^{q^4+q} + x_4\gamma^q + a^q\gamma^{q^2} - a^q\gamma^{q^3})\cdot\\&& (x_2x_5+x_2a^{q^2}+x_5a^{q^5}+ \gamma^{q^2} - \gamma^{q^3} + \gamma^{q^4})\cdot r_1(x_0,x_1,x_2,x_3,x_4,x_5)\cdot r_2(x_0,x_1,x_2,x_3,x_4,x_5)\cdot \\&&r_3(x_0,x_1,x_2,x_3,x_4,x_5).
\end{eqnarray*}
The explicit expression of $r_1$, $r_2$ and $r_3$ are given in Section \ref{Appendice}.
To see that $r_{5,3}$ is not the zero polynomial, we observe that the coefficient of $
x_2^5x_4^5x_5^7$ in $r_{5,3}$ equals
\begin{equation}\label{coefficiente}
A(a-a^{q^2})^{q^2}(a-a^{q^4})^q(a^{q^5+1}+a^{q^4+q^3}-a^{q^5+q^4}-\gamma+\gamma^q-\gamma^{q^2})(a^{q+1}+a^{q^3+q^2}-a^{q^2+q}-\gamma+\gamma^q-\gamma^{q^2}).
\end{equation}
The first three factors in \eqref{coefficiente} cannot vanish, otherwise $A=0$, a contradiction. On the other hand, assume $a^{q^5+1}+a^{q^4+q^3}-a^{q^5+q^4}-\gamma+\gamma^q-\gamma^{q^2}=0$. Since $-\gamma+\gamma^q-\gamma^{q^2}=-\gamma^{q^3}+\gamma^{q^4}+\gamma^{q^5}$, we have 
$$
a^{q+1}+a^{q^3+q^2}-a^{q^2+q}-\gamma+\gamma^q-\gamma^{q^2}=(a^{q^5+1}+a^{q^4+q^3}-a^{q^5+q^4}-\gamma+\gamma^q-\gamma^{q^2})^{q^3}=0.
$$
Taking the sum of $a^{q^5+1}+a^{q^4+q^3}-a^{q^5+q^4}-\gamma+\gamma^q-\gamma^{q^2}=0$ and $a^{q+1}+a^{q^3+q^2}-a^{q^2+q}-\gamma+\gamma^q-\gamma^{q^2}=0$, we obtain $A=0$, a contradiction. Therefore, $r_{5,3}$ is not the zero polynomial, and hence $\mathcal{W}$ is a variety of dimension at most $2$ properly contained in $\mathcal{V}$. The claim follows arguing as in the proof of Proposition \ref{Propmain1}.
\end{proof}

Combining Propositions \ref{Propmain1}, \ref{Propmain2}, and requiring that $m\in\Fqs$ satisfies $m^{q+1}+m^qa^q+ma+b^q+b\neq 0$, we obtain the following result.

\begin{theorem}\label{ultimosez1}
Through any point $(a,b)\in \AGqs$ such that $A\neq 0$ there are at least
$$
q^3-3422q^{5/2}-5\cdot 60^{13/3}q^2-9q^2-(q^2+1499\cdot 1500q^{1/2}+5\cdot 1500^{13/3}q)-(q+1)
$$
$\Fqs$-rational $(q+1)$-secants to $\mathcal{H}_q(\Fqs)$.
\end{theorem}

Theorem \ref{mainsec1} follows from Theorem \ref{ultimosez1}.
\subsection{The case $A=0$}

Now we consider the case $A=0$. In this section, we restrict our search to $m\in \mathbb{F}_{q^2}$ showing the existence of a suitable choice of such $m$ for which the polynomial $f(x)$ has $q+1$ distinct roots in $\mathbb{F}_{q^6}$. So, our aim is to prove the following.

\begin{theorem}\label{mainsec2}
Through any point $(a,b)\in \AGqs$ such that $A=0$ there exists at least one $\Fqs$-rational $(q+1)$-secant to $\mathcal{H}_q(\Fqs)$ of affine equation $Y=m(X-a)+b$, with $m\in \mathbb{F}_{q^2}$.
\end{theorem}

First, we point out that for points $(a,b)$ lying in $\mathrm{AG}(2,q^2)$, Theorem \ref{mainsec2} is well-known, see \cite[Lemma 7.20]{MR1612570}. For this reason, henceforth we will restrict our attention to the case $(a,b)\not\in\mathrm{AG}(2,q^2)$.


We follow the same approach of the previous section, with the additional assumption that $m\in \mathbb{F}_{q^2}$. In this case, the polynomial $h(X):=X_{q^6}-X$ reads 
\begin{equation}\label{Eq:h}
    h(X)=\frac{g(m,m^q)(X-m^q)}{(m^{q+1}+ma+m^qa^{q}+b+b^{q})^{q^4+q^2+1}},
\end{equation}
where 
\begin{equation}\label{Eq:quartica}
g(y_0,y_1)=-C^qy_0y_1^2-D^qy_1^2+Cy_0^2y_1+Ey_0y_1-F^qy_1+Dy_0^2+Fy_0+B.
\end{equation}
and
\begin{eqnarray*}
B&=&(b^{q^3+q+1}+b^{q^4+q+1}+b^{q^3+q^2+1}-b^{q^5+q^2+1}-b^{q^4+q^3+1}-b^{q^5+q^3+1}-b^{q^4+q^2+q}-\\&&-b^{q^5+q^2+q}-b^{q^4+q^3+q}+b^{q^4+q^4+q} +b^{q^5+q^3+q^2}+b^{q^5+q^4+q^2}),\\
C&=&a^{q^4+q+1}+a^{q^3+q^2+1}-a^{q^5+q^2+1}- a^{q^4+q^3+1}+ab^q-ab^{q^5}-a^{q^4+q^2+q}+a^{q^5+q^4+q^2}-\\&&-a^{q^2}b^q+a^{q^2}b^{q^3}-a^{q^4}b^{q^3}+a^{q^4}b^{q^5}\\
D&=&a^{q^2+1}b^{q^3}-a^{q^2+1}b^{q^5}+a^{q^4+1}b^{q}-a^{q^4+1}b^{q^3}-a^{q^4+q^2}b^{q}+a^{q^4+q^2}b^{q^5}\\
E&=&a^{q+1}b^{q^3}+a^{q+1}b^{q^4}+a^{q^3+1}b^{q}+ a^{q^3+1}b^{q^2}-a^{q^3+1}b^{q^4}-a^{q^3+1}b^{q^5}- a^{q^5+1}b^{q^2}-\\&&-a^{q^5+1}b^{q^3}-a^{q^2+q}b^{q^4}-a^{q^2+q}b^{q^5} +a^{q^4+q}b-a^{q^4+q}b^{q^2}-a^{q^4+q}b^{q^3}+a^{q^4+q}b^{q^5}+a^{q^3+q^2}b+\\&& +a^{q^3+q^2}b^{q^5}-a^{q^5+q^2}b-a^{q^5+q^2}b^{q}+a^{q^5+q^2}b^{q^3}+a^{q^5+q^2}b^{q^4}-a^{q^4+q^3}b- a^{q^4+q^3}b^{q}+\\&& +a^{q^5+q^4}b^{q}+a^{q^5+q^4}b^{q^2}+b^{q+1}-b^{q^5+1}-b^{q^2+q} +b^{q^3+q^2}-b^{q^4+q^3}+b^{q^5+q^4},\\
F&=&ab^{q^3+q}+ab^{q^4+q}+ab^{q^3+q^2}-ab^{q^5+q^2}-ab^{q^4+q^3}-ab^{q^5+q^3}+a^{q^2}b^{q^3+1}-a^{q^2}b^{q^5+1}-\\&&-a^{q^2}b^{q^4+q}-a^{q^2}b^{q^5+q}+a^{q^2}b^{q^5+q^3}+a^{q^2}b^{q^5+q^4}+a^{q^4}b^{q+1}-a^{q^4}b^{q^3+1}-a^{q^4}b^{q^2+q}-\\&&-a^{q^4}b^{q^3+q}+a^{q^4}b^{q^5+q}+a^{q^4}b^{q^5+q^2}
\end{eqnarray*}

\begin{comment}
Since $h(x)$ is the zero polynomial, we must have $g(a,b,m)=0$. Viceversa, if $g(a,b,m)=0$ the roots of $f(X)$ all lie in $\Fqs$. So, because of the correspondence between the intersections of $\cH$ and $\ell$ in $\PGqs$ and the roots of $f(X)$ in $\Fqs$, we have the following.
\end{comment}

Also, observe that if $m^{q+1}+m^qa^q+ma+b^q+b=0$ and $m\in \mathbb{F}_{q^2}$, then either $(a,b)\in \mathrm{AG}(2,\mathbb{F}_{q^2})$ or $m=-(b^{q^2}-b)/(a^{q^2}-a)$. Therefore, from Proposition \ref{prop2} we have the following.

\begin{proposition}\label{prop2bis}
Let $(a,b)\in \AGqs\setminus \mathrm{AG}(2,\mathbb{F}_{q^2})$ and $m\in \mathbb{F}_{q^2}$ such that $m\neq -(b^{q^2}-b)/(a^{q^2}-a)$. Then, the line $Y=m(X-a)+b$ intersects $\cH_q$ in exactly $q+1$ distinct points in $\PGqs$ if and only if $g(m,m^q)=0$. 
\end{proposition}

By Proposition \ref{prop2bis}, in order to prove Theorem \ref{mainsec2}, for each $(a,b)\in \AGqs\setminus \mathrm{AG}(2,\mathbb{F}_{q^2})$ such that $A=0$ we must find a suitable $m\in \mathbb{F}_{q^2}$ such that $g(m,m^q)=0$ and $m\neq -(b^{q^2}-b)/(a^{q^2}-a)$.

\begin{proposition}
Let $(a,b)\in \AGqs\setminus \mathrm{AG}(2,\mathbb{F}_{q^2})$ such that $A= 0$. Then, there is at least one value $m\in \mathbb{F}_{q^2}$ such that $g(m,m^q)=0$ and $m\neq -(b^{q^2}-b)/(a^{q^2}-a)$.
\end{proposition}
\begin{proof}
Denote by $\{\xi,\xi^{q}\}$ a normal basis of $\mathbb{F}_{q^2}$ over $\mathbb{F}_q$, and $\eta$ be the change of variables given by
$$y_0=x_0+\xi x_1, \quad  y_1=x_1+\xi x_0.$$
Also, let $\overline{g}(x_0,x_1)=g\circ \eta$. As it can be easily checked $E^q=-E$, and hence the polynomial $\overline{g}(x_0,x_1)$ is fixed by the Frobenius automorphism $\psi$
$$ \sum_{(i_0,i_1)} \alpha_{(i_0,i_1)}x_0^{i_0}x_1^{i_1}\mapsto  \sum_{(i_0,i_1)} \alpha_{(i_0,i_1)}^q x_0^{i_0}x_1^{i_1}.$$
Therefore the algebraic curve $\cQ$ with affine equation $\overline{g}(x_0,x_1)=0$ is defined over $\mathbb{F}_{q}$. First, assume that $C\neq 0$ holds. Then, $\cQ$ has degree $3$ and hence it is either absolutely irreducible or it contains an absolutely irreducible component fixed by $\psi$. Thus, by the Hasse-Weil bound there exist at least $q+1-2\sqrt{q}$ $\mathbb{F}_q$-rational points $(\bar{x}_0,\bar{x}_1)$ of $\cQ$. By construction, such points correspond to values $m\in \mathbb{F}_{q^2}$ satisfying $g(m,m^q)=0$, whence the claim is proved. 

Thus, from now on assume that $C=0$ holds.  
Then, from $A=0$ and $a\not\in\mathbb{F}_{q^2}$ we obtain 
\begin{equation}\label{aq5}
a^{q^5}=\frac{a^{q+1} - a^{q^2+q} + a^{q^3+q^2} - a^{q^4+q^3}}{a^{q^4}-a}.
    \end{equation}
Combining this condition with $C=0$, we have
\begin{eqnarray*}
&&a^{q^2+q+1}-a^{q^4+q+1}-a^{q^3+q^2+1}+a^{q^4+q^3+1}+ab^{q^2}-ab^{q^4}-a^{2q^2+q}\\&&+a^{q^4+q^2+q}+a^{q^3+2q^2}-a^{q^4+q^3+q^2}-a^{q^2}b+a^{q^2}b^{q^4}+a^{q^4}n-a^{q^4}b^{q^2}=0,
\end{eqnarray*}
and hence 
\begin{eqnarray}\label{C=0bis}
b^{q^4}&=&\frac{(a-a^{q^4})b^{q^2}+(a^{q^4}-a^{q^2})b+(a-a^{q^2})^{q^2+q+1}}{a-a^{q^2}}.
\end{eqnarray}
Substituting \eqref{aq5} and \eqref{C=0bis}, a direct computation shows that $g(m,m^q)=0$ reads
$$
(a-a^{q^2})^{q^2}\cdot g_1(m,m^q)\cdot g_2(m,m^q)=0
$$
with
\begin{eqnarray*}
g_1(m,m^q)&=& m^q(a-a^{q^2})^q(b^{q^2}-b+a^{q+1}-a^{q^2+q})+m(a-a^{q^2})(b^{q^2}-b+a^{q+1}-a^{q^2+q})^q+\\&&
(a^{q+1}b^{q}+a^{q+1}b^{q^2}-a^{q^3+1}b^{q}-a^{q^3+1}b^{q^2}-a^{q^2+q}b^{q}-
a^{q^2+q}b^{q^2}+a^{q^3+q^2}b^{q}+\\&&+a^{q^3+q^2}b^{q^2}-b^{q+1}+b^{q^3+1}+b^{q^
2+q}-b^{q^3+q^2}),
\end{eqnarray*}
\begin{eqnarray*}
g_2(m,m^q)&=& m^q(a-a^{q^2})^{q+1}+m(a-a^{q^2})(a-a^{q^4})+\\&&
-a^{q^2+q+1}+a^{q^4+q+1}+a^{q^3+q^2+1}-a^{q^4+q^3+1}+ab+ab^{q}-ab^{q^2}-ab^{
q^3}\\&&+a^{2q^2+q}-a^{q^4+q^2+q}-a^{q^3+2q^2}+a^{q^4+q^3+q^2}-a^{q^2}b^{q}+a
^{q^2}b^{q^3}-a^{q^4}b+a^{q^4}b^{q^2}.
\end{eqnarray*}
Therefore, there are $q$ distinct values $m\in \mathbb{F}_{q^2}$ satisfying $g_2(m,m^q)=0$,  and each of them (apart at most one, namely $m=-(b^{q^2}-b)/(a^{q^2}-a)$) satisfies $g(m,m^q)=0$.

\end{proof}

\section{Acknowledgements} 

This research was supported by the Italian National Group for Algebraic and Geometric Structures and their Applications (GNSAGA - INdAM). The second author is funded by the project ``Metodi matematici per la firma digitale ed il cloud computing" (Programma Operativo Nazionale (PON) “Ricerca e Innovazione” 2014-2020, University of Perugia).

\bibliographystyle{abbrv}
\bibliography{biblio.bib}

\section{Appendix}\label{Appendice}
Let ${\bf y}=(y_0,\ldots,y_5)$. We define the following polynomials:
\begin{align*}
f_1({\bf y}) :=&
-y_0^2 y_1 y_2 y_3 + y_0^2 y_1 y_2 y_5 - y_0^2 y_1 y_3 a^{q^3} - y_0^2 y_1 y_4 y_5 - y_0^2 y_1 y_4 a^{q^4}- y_0^2 y_1\gamma^{q^3}\\
& - y_0^2 y_2 y_3 a^{q^2} + y_0^2 y_2 y_5 a^{q^2} + y_0^2 y_3 y_4 y_5 + y_0^2 y_3 y_4 a^{q^4} + y_0^2 y_3 y_5 a^{q^3} - y_0^2 y_3 \gamma^{q^2}\\& + y_0^2 y_3\gamma^{q^3} + y_0^2 y_5 \gamma^{q^2} + y_0 y_1 y_2 y_3 y_4- y_0 y_1 y_2 y_3 a^{q} + y_0 y_1 y_2 y_4 a^{q^4} + y_0 y_1 y_2 y_5 a^{q}\\
& + y_0 y_1 y_2 y_5 a^{q^5}+ y_0 y_1 y_2\gamma^{q^4} + y_0 y_1 y_3 y_4 a^{q^3} -y_0 y_1 y_3 a^{q^3+q} - y_0 y_1 y_4 y_5 a^q - y_0 y_1 y_4 y_5 a^{q^5}\\
&- y_0 y_1 y_4 a^{q^4+q} + y_0 y_1 y_4\gamma^{q^3} - y_0 y_1 y_4\gamma^{q^4} - y_0 y_1 a^q\gamma^{q^3} - y_0 y_2 y_3 y_4 y_5+y_0 y_2 y_3 y_4 a^{q^2}\\
& - y_0 y_2 y_3 y_4 a^{q^4} - y_0 y_2 y_3 y_5 a^{q^3} - y_0 y_2 y_3 a^{q^3+q^2}-y_0 y_2 y_3 \gamma^q + y_0 y_2 y_3 \gamma^{q^2} - y_0 y_2 y_3\gamma^{q^3}\\
& - y_0 y_2 y_4 y_5 a^{q^2} + y_0 y_2 y_5 a^{q^5+q^2}+ y_0 y_2 y_5 \gamma^q - y_0 y_2 y_5 \gamma^{q^2} - y_0 y_2 a^{q^2}\gamma^{q^3} + y_0 y_2 a^{q^2}\gamma^{q^4} \\
&+y_0 y_3 y_4 y_5 a^{q^5}+ y_0 y_3 y_4 a^{q^3} a^{q^4} + y_0 y_3 y_4 \gamma^{q^2} - y_0 y_3 y_4\gamma^{q^3} + y_0 y_3 y_4\gamma^{q^4} + y_0 y_3 y_5a^{q^5+q^3}\\
&- y_0 y_3 a^{q^3} \gamma^q + y_0 y_3 a^{q^3}\gamma^{q^4} - y_0 y_4 y_5 \gamma^q -y_0 y_4 a^{q^4} \gamma^q + y_0 y_4 a^{q^4} \gamma^{q^2}+ y_0 y_5 a^{q^5} \gamma^{q^2}\\
& - y_0 \gamma^{q^3+q} + y_0 \gamma^{q^4+q^2} +y_1 y_2 y_3 y_4 a^q + y_1 y_2 y_4 a^{q^4+q}+ y_1 y_2 y_5 a^{q^5+q} + y_1 y_2 a^q\gamma^{q^4}\\
& + y_1 y_3 y_4 a^{q^3+q} - y_1 y_4 y_5 a^{q^5+q} + y_1 y_4 a^q\gamma^{q^3}- y_1 y_4 a^q\gamma^{q^4} - y_2 y_3 y_4 y_5 a^{q^5} + y_2 y_3 y_4 a^{q^3+q^2}\\
& - y_2 y_3 y_4 a^{q^4+q^3} + y_2 y_3 y_4 \gamma^q-y_2 y_3 y_4 \gamma^{q^2} + y_2 y_3 y_4\gamma^{q^3} - y_2 y_3 y_4\gamma^{q^4} - y_2 y_3 y_5a^{q^5+q^3}\\
& - y_2 y_3 a^{q^3}\gamma^{q^4}-y_2 y_4 y_5 a^{q^5+q^2} + y_2 y_4 a^{q^2}\gamma^{q^3} - y_2 y_4 a^{q^2}\gamma^{q^4} + y_2 y_4 a^{q^4} \gamma^q - y_2 y_4 a^{q^4} \gamma^{q^2}\\
&+y_2 y_5 a^{q^5} \gamma^q - y_2 y_5 a^{q^5} \gamma^{q^2} + y_2 \gamma^{q^4+q} - y_2 \gamma^{q^4+q^2} + y_3 y_4 a^{q^3} \gamma^q -y_4 y_5 a^{q^5} \gamma^q\\
&+ y_4 \gamma^{q^3+q} - y_4 \gamma^{q^4+q};
\end{align*}

\begin{align*}
f_2({\bf y}):=&-2 y_0^2 y_1 y_2 y_3 a + 2 y_0^2 y_1 y_2 y_5 a - 2 y_0^2 y_1 y_3 a^{q^3+1}- 2 y_0^2 y_1 y_4 y_5 a - 2 y_0^2 y_1 y_4 a^{q^4+1}\\
& - 2 y_0^2 y_1 a\gamma^{q^3} -2 y_0^2 y_2 y_3 a^{q^2+1}+ 2 y_0^2 y_2 y_5 a^{q^2+1}+ 2 y_0^2 y_3 y_4 y_5 a + 2 y_0^2 y_3 y_4 a^{q^4+1}\\
& +2 y_0^2 y_3 y_5 a^{q^3+1} - 2 y_0^2 y_3 a \gamma^{q^2} +2 y_0^2 y_3 a\gamma^{q^3} + 2 y_0^2 y_5 a \gamma^{q^2}+ y_0 y_1 y_2 y_3 y_4 a \\
&+y_0 y_1 y_2 y_3 y_4 a^{q^4} + y_0 y_1 y_2 y_3 y_5 a^{q^3}- y_0 y_1 y_2 y_3 y_5 a^{q^5} -y_0 y_1 y_2 y_3 a^{q+1} - y_0 y_1 y_2 y_3 a^{q^2+q}\\
& + y_0 y_1 y_2 y_3 a^{q^3+q^2} -y_0 y_1 y_2 y_3 \gamma + y_0 y_1 y_2 y_3 \gamma^q- y_0 y_1 y_2 y_3 \gamma^{q^2} + y_0 y_1 y_2 y_3\gamma^{q^3}\\
& -y_0 y_1 y_2 y_3\gamma^{q^5} + y_0 y_1 y_2 y_4 y_5 a^{q^2}- y_0 y_1 y_2 y_4 y_5 a^{q^4} +y_0 y_1 y_2 y_4 a^{q^4+1} + y_0 y_1 y_2 y_4 a^{q^4+q^2}\\
&+ y_0 y_1 y_2 y_5 a^{q+1}+ y_0 y_1 y_2 y_5 a^{q^5+1} + y_0 y_1 y_2 y_5 a^{q^2+q} + y_0 y_1 y_2 y_5 \gamma - y_0 y_1 y_2 y_5 \gamma^q\\
&+ y_0 y_1 y_2 y_5 \gamma^{q^2} - y_0 y_1 y_2 y_5\gamma^{q^4} + y_0 y_1 y_2 y_5\gamma^{q^5} + y_0 y_1 y_2 a\gamma^{q^4} +y_0 y_1 y_2 a^{q^2}\gamma^{q^3}\\
& + y_0 y_1 y_3 y_4 y_5 a^q - y_0 y_1 y_3 y_4 y_5 a^{q^3}+y_0 y_1 y_3 y_4 a^{q^3+1}+ y_0 y_1 y_3 y_4 a^{q^4+q} + y_0 y_1 y_3 y_5 a^{q^3+q}\\
& -y_0 y_1 y_3 y_5a^{q^5+q^3} - y_0 y_1 y_3 a^{q^3+q+1}- y_0 y_1 y_3 a^q \gamma^{q^2} + y_0 y_1 y_3 a^q\gamma^{q^3}- y_0 y_1 y_3 a^{q^3}\gamma\\
& + y_0 y_1 y_3 a^{q^3}\gamma^q- y_0 y_1 y_3 a^{q^3}\gamma^{q^5} - y_0 y_1 y_4 y_5 a^{q+1} -y_0 y_1 y_4 y_5 a^{q^5+1} - y_0 y_1 y_4 y_5 a^{q^5+q^4}\\
&- y_0 y_1 y_4 y_5 \gamma + y_0 y_1 y_4 y_5 \gamma^q - y_0 y_1 y_4 y_5\gamma^{q^3} + y_0 y_1 y_4 y_5\gamma^{q^4}- y_0 y_1 y_4 y_5\gamma^{q^5}\\
& - y_0 y_1 y_4 a^{q^4+q+1} +y_0 y_1 y_4 a\gamma^{q^3} - y_0 y_1 y_4 a\gamma^{q^4}- y_0 y_1 y_4 a^{q^4}\gamma + y_0 y_1 y_4 a^{q^4}\gamma^q\\
& - y_0 y_1 y_4 a^{q^4}\gamma^{q^5} + y_0 y_1 y_5 a^q \gamma^{q^2}- y_0 y_1 y_5 a^{q^5}\gamma^{q^3} - y_0 y_1 a^{q+1}\gamma^{q^3} - y_0 y_1 \gamma^{q^3+1}\\
& + y_0 y_1 \gamma^{q^3+q} - y_0 y_1\gamma^{q^5+q^3}- y_0 y_2 y_3 y_4 y_5 a - y_0 y_2 y_3 y_4 y_5 a^{q^2} + y_0 y_2 y_3 y_4 a^{q^2+1}\\
&- y_0 y_2 y_3 y_4 a^{q^4+1}- y_0 y_2 y_3 y_5 a^{q^3+1} - y_0 y_2 y_3 y_5 a^{q^5+q^2} - y_0 y_2 y_3 a^{q^3q^2+1} - y_0 y_2 y_3 a \gamma^q\\
&+ y_0 y_2 y_3 a \gamma^{q^2} - y_0 y_2 y_3 a\gamma^{q^3} - y_0 y_2 y_3 a^{q^2} \gamma - y_0 y_2 y_3 a^{q^2}\gamma^{q^5}- y_0 y_2 y_4 y_5 a^{q^2+1} \\
& -y_0 y_2 y_4 y_5 a^{q^4+q^2}+ y_0 y_2 y_5 a^{q^5+q^2+1} + y_0 y_2 y_5 a \gamma^q- y_0 y_2 y_5 a \gamma^{q^2} + y_0 y_2 y_5 a^{q^2} \gamma\\
& - y_0 y_2 y_5 a^{q^2}\gamma^{q^4} + y_0 y_2 y_5 a^{q^2}\gamma^{q^5}- y_0 y_2 a^{q^2+1}\gamma^{q^3} + y_0 y_2 a^{q^2+1}\gamma^{q^4} + y_0 y_3 y_4 y_5 a^{q^5+1}\\
& - y_0 y_3 y_4 y_5 a^{q^4+q^3}+ y_0 y_3 y_4 y_5 a^{q^5+q^4} + y_0 y_3 y_4 y_5 \gamma - y_0 y_3 y_4 y_5 \gamma^{q^2} + y_0 y_3 y_4 y_5\gamma^{q^3}\\
&- y_0 y_3 y_4 y_5\gamma^{q^4} + y_0 y_3 y_4 y_5\gamma^{q^5} + y_0 y_3 y_4 a^{q^4+q^3+1} + y_0 y_3 y_4 a \gamma^{q^2}- y_0 y_3 y_4 a\gamma^{q^3} \\
&+ y_0 y_3 y_4 a\gamma^{q^4} + y_0 y_3 y_4 a^{q^4}\gamma + y_0 y_3 y_4 a^{q^4}\gamma^{q^5}+ y_0 y_3 y_5 a^{q^5+q^3+1} + y_0 y_3 y_5 a^{q^3}\gamma \\
&- y_0 y_3 y_5 a^{q^3}\gamma^{q^4} + y_0 y_3 y_5 a^{q^3}\gamma^{q^5}- y_0 y_3 y_5 a^{q^5} \gamma^{q^2} + y_0 y_3 y_5 a^{q^5}\gamma^{q^3} - y_0 y_3 a^{q^3+1}b^q\\
& + y_0 y_3 a^{q^3+1}\gamma^{q^4}- y_0 y_3 \gamma^{q^2+1}+ y_0 y_3 \gamma^{q^3+1} - y_0 y_3 \gamma^{q^5+q^2} + y_0 y_3\gamma^{q^5+q^3} - y_0 y_4 y_5 a \gamma^q\\
&- y_0 y_4 y_5 a^{q^4} \gamma^{q^2} - y_0 y_4 a^{q^4+1} \gamma^q + y_0 y_4 a^{q^4+1} \gamma^{q^2} + y_0 y_5 a^{q^5+1} \gamma^{q^2}+ y_0 y_5 \gamma^{q^2+1}\\
&- y_0 y_5 \gamma^{q^4+q^2} + y_0 y_5 \gamma^{q^5+q^2} - y_0 a \gamma^{q^3+q} + y_0 a \gamma^{q^4+q^2}- y_1 y_2 y_3 y_4 y_5 a^q + y_1 y_2 y_3 y_4 y_5 a^{q^5}\\
& + y_1 y_2 y_3 y_4 a^{q^2+q} - y_1 y_2 y_3 y_4 a^{q^3+q^2}+ y_1 y_2 y_3 y_4 a^{q^4+q^3} + y_1 y_2 y_3 y_4 \gamma - y_1 y_2 y_3 y_4 \gamma^q\\
& + y_1 y_2 y_3 y_4 \gamma^{q^2}- y_1 y_2 y_3 y_4\gamma^{q^3} + y_1 y_2 y_3 y_4\gamma^{q^4} - y_1 y_2 y_3 y_5 a^{q^5+q} +y_1 y_2 y_3 y_5a^{q^5+q^3}\\
&- y_1 y_2 y_3 a^q\gamma^{q^5} + y_1 y_2 y_3 a^{q^3}\gamma^{q^4} - y_1 y_2 y_4 y_5 a^{q^4+q} + y_1 y_2 y_4 y_5 a^{q^5+q^2}+ y_1 y_2 y_4 a^{q^4+q^2+q}\displaybreak\\
& - y_1 y_2 y_4 a^{q^2}\gamma^{q^3} + y_1 y_2 y_4 a^{q^2}\gamma^{q^4} + y_1 y_2 y_4 a^{q^4} \gamma- y_1 y_2 y_4 a^{q^4} \gamma^q + y_1 y_2 y_4 a^{q^4} \gamma^{q^2}\\
& + y_1 y_2 y_5 a^{q^5+q^2+q} - y_1 y_2 y_5 a^q\gamma^{q^4}+ y_1 y_2 y_5 a^q\gamma^{q^5} + y_1 y_2 y_5 a^{q^5} \gamma - y_1 y_2 y_5 a^{q^5} \gamma^q \\
&+ y_1 y_2 y_5 a^{q^5} \gamma^{q^2}+ y_1 y_2 a^{q^2+q}\gamma^{q^4} + y_1 y_2 \gamma^{q^4+1} - y_1 y_2 \gamma^{q^4+q} + y_1 y_2 \gamma^{q^4+q^2} - y_1 y_3 y_4 y_5 a^{q^3+q}\\
&+ y_1 y_3 y_4 y_5 a^{q^5+q} + y_1 y_3 y_4 a^{q^4+q^3+q} + y_1 y_3 y_4 a^q \gamma^{q^2} - y_1 y_3 y_4 a^q\gamma^{q^3}+ y_1 y_3 y_4 a^q\gamma^{q^4}\\
& + y_1 y_3 y_4 a^{q^3}\gamma - y_1 y_3 y_4 a^{q^3}\gamma^q + y_1 y_3 a^{q^3+q}\gamma^{q^4}- y_1 y_3 a^{q^3+q}\gamma^{q^5} - y_1 y_4 y_5 a^{q^5+q^4+q}\\
& - y_1 y_4 y_5 a^q\gamma^{q^3} + y_1 y_4 y_5 a^q\gamma^{q^4}\- y_1 y_4 y_5 a^q\gamma^{q^5} - y_1 y_4 y_5 a^{q^5} \gamma + y_1 y_4 y_5 a^{q^5} \gamma^q + y_1 y_4 a^{q^4+q} \gamma^{q^2}\\
&- y_1 y_4 a^{q^4+q}\gamma^{q^5} + y_1 y_4 \gamma^{q^3+1} - y_1 y_4 \gamma^{q^4+1} - y_1 y_4 \gamma^{q^3+q} + y_1 y_4 \gamma^{q^4+q}+ y_1 y_5 a^{q^5+q} \gamma^{q^2}\\
& -y_1 y_5 a^{q^5+q}\gamma^{q^3} + y_1 a^q \gamma^{q^4+q^2} - y_1 a^q\gamma^{q^5+q^3} - y_2 y_3 y_4 y_5 a^{q^3+q^2}+ y_2 y_3 y_4 y_5 a^{q^4+q^3}\\
&  - y_2 y_3 y_4 y_5 a^{q^5+q^4}- y_2 y_3 y_4 y_5 \gamma^q + y_2 y_3 y_4 y_5 \gamma^{q^2}- y_2 y_3 y_4 y_5\gamma^{q^3} + y_2 y_3 y_4 y_5\gamma^{q^4} - y_2 y_3 y_4 y_5\gamma^{q^5}\\
& + y_2 y_3 y_4 a^{q^2} \gamma- y_2 y_3 y_4 a^{q^4}\gamma^{q^5} - y_2 y_3 y_5a^{q^5+q^3+q^2} + y_2 y_3 y_5 a^{q^3}\gamma^{q^4} - y_2 y_3 y_5 a^{q^3}\gamma^{q^5}\\
&- y_2 y_3 y_5 a^{q^5} \gamma^q + y_2 y_3 y_5 a^{q^5} \gamma^{q^2} - y_2 y_3 y_5 a^{q^5}\gamma^{q^3} - y_2 y_3 a^{q^3+q^2}\gamma^{q^5}- y_2 y_3 \gamma^{q^5+q} + y_2 y_3 \gamma^{q^5+q^2}\\
& - y_2 y_3\gamma^{q^5+q^3} - y_2 y_4 y_5 a^{q^5+q^4+q^2} - y_2 y_4 y_5 a^{q^2}\gamma^{q^3}+ y_2 y_4 y_5 a^{q^2}\gamma^{q^4} - y_2 y_4 y_5 a^{q^2}\gamma^{q^5} - y_2 y_4 y_5 a^{q^4} \gamma^q\\
& + y_2 y_4 y_5 a^{q^4} \gamma^{q^2}+ y_2 y_4 a^{q^4+q^2}\gamma - y_2 y_4 a^{q^4+q^2}\gamma^{q^5} + y_2 y_5 a^{q^5+q^2} \gamma - y_2 y_5 a^{q^5+q^2}\gamma^{q^3}- y_2 y_5 \gamma^{q^4+q}\\
& + y_2 y_5 \gamma^{q^5+q} + y_2 y_5 \gamma^{q^4+q^2} - y_2 y_5 \gamma^{q^5+q^2} +y_2 a^{q^2} \gamma^{q^4+1} - y_2 a^{q^2}\gamma^{q^5+q^3}- y_3 y_4 y_5 a^{q^3} \gamma^q\\
& + y_3 y_4 y_5 a^{q^5} \gamma + y_3 y_4 a^{q^4+q^3} \gamma + y_3 y_4 \gamma^{q^2+1}- y_3 y_4 \gamma^{q^3+1}+ y_3 y_4 \gamma^{q^4+1} + y_3 y_5 a^{q^5+q^3} \gamma\\
& - y_3 y_5 a^{q^5+q^3} \gamma^q + y_3 a^{q^3}\gamma^{q^4+1} - y_3 a^{q^3} \gamma^{q^5+q}- y_4 y_5 a^{q^5+q^4} \gamma^{q} - y_4 y_5 \gamma^{q^3+q} + y_4 y_5 \gamma^{q^4+q}\\
& - y_4 y_5 \gamma^{q^5+q} + y_4 a^{q^4} \gamma^{q^2+1}-y_4 a^{q^4} \gamma^{q^5+q} + y_5 a^{q^5} \gamma^{q^2+1} - y_5 a^{q^5} \gamma^{q^3+q} + \gamma^{q^4+q^2+1} - \gamma^{q^5+q^3+q};
\end{align*}

\begin{align*}
f_3({\bf y}):=&-y_0^2 y_1 y_2 y_3 a^2 + y_0^2 y_1 y_2 y_5 a^2 - y_0^2 y_1 y_3 a^{q^3+2}- y_0^2 y_1 y_4 y_5 a^2 - y_0^2 y_1 y_4 a^{q^4+2} - y_0^2 y_1 a^2\gamma^{q^3}\\
& - y_0^2 y_2 y_3 a^{q^2+2}+ y_0^2 y_2 y_5 a^{q^2+2} + y_0^2 y_3 y_4 y_5 a^{2}+ y_0^2 y_3 y_4 a^{q^4+2}+ y_0^2 y_3 y_5 a^{q^3+2} - y_0^2 y_3 a^{2}\gamma^{q^2}\\
& + y_0^2 y_3 a^{2}\gamma^{q^3} + y_0^2 y_5 a^{2}\gamma^{q^2}+ y_0 y_1 y_2 y_3 y_4 a^{q^4+1} + y_0 y_1 y_2 y_3 y_5 a^{q^3+1} - y_0 y_1 y_2 y_3 y_5 a^{q^5+1}\\
&- y_0 y_1 y_2 y_3 a^{q^2+q+1} + y_0 y_1 y_2 y_3 a^{q^3+q^2+1} - y_0 y_1 y_2 y_3 a \gamma+ y_0 y_1 y_2 y_3 a \gamma^q - y_0 y_1 y_2 y_3 a \gamma^{q^2}\\
& + y_0 y_1 y_2 y_3 a\gamma^{q^3}- y_0 y_1 y_2 y_3 a\gamma^{q^5}+ y_0 y_1 y_2 y_4 y_5 a^{q^2+1}  - y_0 y_1 y_2 y_4 y_5 a^{q^4+1}+ y_0 y_1 y_2 y_4 a^{q^4+q^2+1}\\
& + y_0 y_1 y_2 y_5 a^{q^2+q+1} + y_0 y_1 y_2 y_5 a \gamma- y_0 y_1 y_2 y_5 a \gamma^q + y_0 y_1 y_2 y_5 a \gamma^{q^2} - y_0 y_1 y_2 y_5 a\gamma^{q^4}\\
&+ y_0 y_1 y_2 y_5 a\gamma^{q^5} + y_0 y_1 y_2 a^{q^2+1}\gamma^{q^3} + y_0 y_1 y_3 y_4 y_5 a^{q+1}- y_0 y_1 y_3 y_4 y_5 a^{q^3+1}+ y_0 y_1 y_3 y_4 a^{q^4+q+1}\\
& + y_0 y_1 y_3 y_5 a^{q^3+q+1}- y_0 y_1 y_3 y_5 a^{q^5+q^3+1} - y_0 y_1 y_3 a^{q+1} \gamma^{q^2} + y_0 y_1 y_3 a^{q+1}\gamma^{q^3} - y_0 y_1 y_3 a^{q^3+1}\gamma\\
&+ y_0 y_1 y_3 a^{q^3+1}\gamma^q - y_0 y_1 y_3 a^{q^3+1}\gamma^{q^5} - y_0 y_1 y_4 y_5 a^{q^5+q^4+1} - y_0 y_1 y_4 y_5 a \gamma+ y_0 y_1 y_4 y_5 a \gamma^q\\
& - y_0 y_1 y_4 y_5 a\gamma^{q^3} + y_0 y_1 y_4 y_5 a\gamma^{q^4} - y_0 y_1 y_4 y_5 a\gamma^{q^5}- y_0 y_1 y_4 a^{q^4+1} \gamma + y_0 y_1 y_4 a^{q^4+1} \gamma^q\displaybreak\\& - y_0 y_1 y_4 a^{q^4+1}\gamma^{q^5} + y_0 y_1 y_5 a^{q+1} \gamma^{q^2}- y_0 y_1 y_5 a^{q^5+1}\gamma^{q^3} - y_0 y_1 a \gamma^{q^3+1} + y_0 y_1 a \gamma^{q^3+q}\\
& - y_0 y_1 a\gamma^{q^5+q^3}- y_0 y_2 y_3 y_4 y_5 a^{q^2+1}- y_0 y_2 y_3 y_5 a^{q^5+q^2+1} - y_0 y_2 y_3 a^{q^2+1}\gamma- y_0 y_2 y_3 a^{q^2+1}\gamma^{q^5}\\
& - y_0 y_2 y_4 y_5 a^{q^4+q^2+1} + y_0 y_2 y_5 a^{q^2+1}\gamma - y_0 y_2 y_5 a^{q^2+1}\gamma^{q^4}+ y_0 y_2 y_5 a^{q^2+1}\gamma^{q^5} - y_0 y_3 y_4 y_5 a^{q^4+q^3+1}\\
& + y_0 y_3 y_4 y_5 a^{q^5+q^4+1}+ y_0 y_3 y_4 y_5 a \gamma - y_0 y_3 y_4 y_5 a \gamma^{q^2} + y_0 y_3 y_4 y_5 a\gamma^{q^3} - y_0 y_3 y_4 y_5 a\gamma^{q^4}\\
&+ y_0 y_3 y_4 y_5 a\gamma^{q^5} + y_0 y_3 y_4 a^{q^4+1} \gamma + y_0 y_3 y_4 a^{q^4+1}\gamma^{q^5} + y_0 y_3 y_5 a^{q^3+1}\gamma- y_0 y_3 y_5 a^{q^3+1}\gamma^{q^4}\\
& + y_0 y_3 y_5 a^{q^3+1}\gamma^{q^5} - y_0 y_3 y_5 a^{q^5+1} \gamma^{q^2} + y_0 y_3 y_5 a^{q^5+1}\gamma^{q^3}- y_0 y_3 a \gamma^{q^2+1}+ y_0 y_3 a \gamma^{q^3+1}\\& - y_0 y_3 a \gamma^{q^5+q^2} + y_0 y_3 a\gamma^{q^5+q^3}- y_0 y_4 y_5 a^{q^4+1} \gamma^{q^2} + y_0 y_5 a \gamma^{q^2+1}- y_0 y_5 a \gamma^{q^4+q^2} + y_0 y_5 a \gamma^{q^5+q^2}\\
&- y_1 y_2 y_3 y_4 y_5 a^{q^2+q} + y_1 y_2 y_3 y_4 y_5 a^{q^3+q^2}- y_1 y_2 y_3 y_4 y_5 a^{q^4+q^3}+ y_1 y_2 y_3 y_4 y_5 a^{q^5+q^4} - y_1 y_2 y_3 y_4 y_5 \gamma\\
& + y_1 y_2 y_3 y_4 y_5 \gamma^q - y_1 y_2 y_3 y_4 y_5 \gamma^{q^2}+ y_1 y_2 y_3 y_4 y_5\gamma^{q^3} - y_1 y_2 y_3 y_4 y_5\gamma^{q^4} + y_1 y_2 y_3 y_4 y_5\gamma^{q^5}\\
& + y_1 y_2 y_3 y_4 a^{q^4}\gamma^{q^5}- y_1 y_2 y_3 y_5 a^{q^5+q^2+q} + y_1 y_2 y_3 y_5a^{q^5+q^3+q^2} - y_1 y_2 y_3 y_5 a^{q^3}\gamma^{q^4}+ y_1 y_2 y_3 y_5 a^{q^3}\gamma^{q^5}\\
& - y_1 y_2 y_3 y_5 a^{q^5} \gamma + y_1 y_2 y_3 y_5 a^{q^5} \gamma^q - y_1 y_2 y_3 y_5 a^{q^5} \gamma^{q^2}+ y_1 y_2 y_3 y_5 a^{q^5}\gamma^{q^3} - y_1 y_2 y_3 a^{q^2+q}\gamma^{q^5}\\
& + y_1 y_2 y_3 a^{q^3+q^2}\gamma^{q^5} - y_1 y_2 y_3 \gamma^{q^5+1}+ y_1 y_2 y_3 \gamma^{q^5+q} - y_1 y_2 y_3 \gamma^{q^5+q^2} + y_1 y_2 y_3\gamma^{q^5+q^3}\\
& - y_1 y_2 y_4 y_5 a^{q^4+q^2+q}+ y_1 y_2 y_4 y_5 a^{q^5+q^4+q^2} + y_1 y_2 y_4 y_5 a^{q^2}\gamma^{q^3} - y_1 y_2 y_4 y_5 a^{q^2}\gamma^{q^4} + y_1 y_2 y_4 y_5 a^{q^2}\gamma^{q^5}\\
&- y_1 y_2 y_4 y_5 a^{q^4} \gamma + y_1 y_2 y_4 y_5 a^{q^4} \gamma^q - y_1 y_2 y_4 y_5 a^{q^4} \gamma^{q^2} + y_1 y_2 y_4 a^{q^4+q^2}\gamma^{q^5}- y_1 y_2 y_5 a^{q^2+q}\gamma^{q^4}\\
& + y_1 y_2 y_5 a^{q^2+q}\gamma^{q^5} + y_1 y_2 y_5 a^{q^5+q^2}\gamma^{q^3} - y_1 y_2 y_5 \gamma^{q^4+1}+ y_1 y_2 y_5 \gamma^{q^5+1} + y_1 y_2 y_5 \gamma^{q^4+q}\\
& - y_1 y_2 y_5 \gamma^{q^5+q} - y_1 y_2 y_5 \gamma^{q^4+q^2} + y_1 y_2 y_5 \gamma^{q^5+q^2}+ y_1 y_2 a^{q^2}\gamma^{q^5+q^3} - y_1 y_3 y_4 y_5 a^{q^4+q^3+q}\\
& + y_1 y_3 y_4 y_5 a^{q^5+q^4+q} - y_1 y_3 y_4 y_5 a^q \gamma^{q^2}+ y_1 y_3 y_4 y_5 a^q\gamma^{q^3} - y_1 y_3 y_4 y_5 a^q\gamma^{q^4} + y_1 y_3 y_4 y_5 a^q\gamma^{q^5}\\
& - y_1 y_3 y_4 y_5 a^{q^3} \gamma+ y_1 y_3 y_4 y_5 a^{q^3} \gamma^q + y_1 y_3 y_4 a^{q^4+q}\gamma^{q^5} - y_1 y_3 y_5 a^{q^3+q}\gamma^{q^4} + y_1 y_3 y_5 a^{q^3+q}\gamma^{q^5}\\
&- y_1 y_3 y_5 a^{q^5+q} \gamma^{q^2} + y_1 y_3 y_5 a^{q^5+q}\gamma^{q^3} - y_1 y_3 y_5 a^{q^5+q^3} \gamma + y_1 y_3 y_5 a^{q^5+q^3} \gamma^q- y_1 y_3 a^q \gamma^{q^5+q^2}\\
& + y_1 y_3 a^q\gamma^{q^5+q^3} - y_1 y_3 a^{q^3} \gamma^{q^5+1} + y_1 y_3 a^{q^3} \gamma^{q^5+q}- y_1 y_4 y_5 a^{q^4+q} \gamma^{q^2} - y_1 y_4 y_5 a^{q^5+q^4} \gamma\\
& + y_1 y_4 y_5 a^{q^5+q^4} \gamma^q - y_1 y_4 y_5 \gamma^{q^3+1}+ y_1 y_4 y_5 \gamma^{q^4+1} - y_1 y_4 y_5 \gamma^{q^5+1} + y_1 y_4 y_5 \gamma^{q^3+q}\\
& - y_1 y_4 y_5 \gamma^{q^4+q} + y_1 y_4 y_5 \gamma^{q^5+q}- y_1 y_4 a^{q^4} \gamma^{q^5+1} + y_1 y_4 a^{q^4} \gamma^{q^5+q} - y_1 y_5 a^q \gamma^{q^4+q^2} + y_1 y_5 a^q \gamma^{q^5+q^2}\\
& - y_1 y_5 a^{q^5} \gamma^{q^3+1}+ y_1 y_5 a^{q^5} \gamma^{q^3+q} - y_1 \gamma^{q^5+q^3+1} + y_1 \gamma^{q^5+q^3+q} - y_2 y_3 y_4 y_5 a^{q^2} \gamma - y_2 y_3 y_5 a^{q^5+q^2} \gamma\\
&- y_2 y_3 a^{q^2} \gamma^{q^5+1} - y_2 y_4 y_5 a^{q^4+q^2}\gamma - y_2 y_5 a^{q^2} \gamma^{q^4+1} + y_2 y_5 a^{q^2} \gamma^{q^5+1}- y_3 y_4 y_5 a^{q^4+q^3} \gamma\\
& + y_3 y_4 y_5 a^{q^5+q^4} \gamma - y_3 y_4 y_5 \gamma^{q^2+1}+ y_3 y_4 y_5 \gamma^{q^3+1}- y_3 y_4 y_5 \gamma^{q^4+1} + y_3 y_4 y_5 \gamma^{q^5+1} + y_3 y_4 a^{q^4} \gamma^{q^5+1}\\
& - y_3 y_5 a^{q^3} \gamma^{q^4+1} + y_3 y_5 a^{q^3} \gamma^{q^5+1}- y_3 y_5 a^{q^5} \gamma^{q^2+1} + y_3 y_5 a^{q^5} \gamma^{q^3+1} - y_3 \gamma^{q^5+q^2+1} + y_3 \gamma^{q^5+q^3+1} \\
&-y_4 y_5 a^{q^4} \gamma^{q^2+1}- y_5 \gamma^{q^4+q^2+1} + y_5 \gamma^{q^5+q^2+1};
\end{align*}
\begin{align*}
    f_4({\bf y}):=& 
y_0^2 y_1 y_2 y_3 a - y_0^2 y_1 y_2 y_5 a + y_0^2 y_1 y_3 a^{q^3+1} + y_0^2 y_1 y_4 y_5 a +y_0^2 y_1 y_4 a^{q^4+1} + y_0^2 y_1 a \gamma^{q^3} +y_0^2 y_2 y_3 a^{q^2+1}
\\&
- y_0^2 y_2 y_5a^{q^2+1} - y_0^2 y_3 y_4 y_5 a - y_0^2 y_3 y_4 a^{q^4+1} - y_0^2 y_3 y_5 a^{q^3+1} + y_0^2 y_3 a \gamma^{q^2} -y_0^2 y_3 a \gamma^{q^3} - y_0^2 y_5 a 
\\&
\gamma^{q^2} - y_0 y_1 y_2 y_3 y_4 y_5 - y_0 y_1 y_2 y_3 y_4 a - y_0 y_1 y_2 y_3 y_4 a^{q^4} - y_0 y_1 y_2 y_3 y_5 a^{q^3} +y_0 y_1 y_2 y_3 a^{q^2+q} - y_0 y_1 y_2 
\\&
y_3 a^{q^3+q^2} + y_0 y_1 y_2 y_3 \gamma - y_0 y_1 y_2 y_3 \gamma^q + y_0 y_1 y_2 y_3 \gamma^{q^2} - y_0 y_1 y_2 y_3 \gamma^{q^3} -y_0 y_1 y_2 y_4 y_5 a^{q^2}\\& 
- y_0 y_1 y_2 y_4 a^{q^4+1}- y_0 y_1 y_2 y_4 a^{q^4+q^2} - y_0 y_1 y_2 y_5 a^{q^5+1} - y_0 y_1 y_2 y_5 a^{q^2+q} - y_0 y_1 y_2 y_5 \gamma +y_0 y_1 y_2 y_5 \gamma^q
\\&
- y_0 y_1 y_2 y_5 \gamma^{q^2} - y_0 y_1 y_2 a \gamma^{q^4} - y_0 y_1 y_2 a^{q^2} \gamma^{q^3} - y_0 y_1 y_3 y_4 y_5 a^q - y_0 y_1 y_3 y_4 a^{q^3+1} -y_0 y_1 y_3 y_4 a^{q^4+q} 
\\&
- y_0 y_1 y_3 y_5 a^{q^3+q} + y_0 y_1 y_3 a^q \gamma^{q^2} - y_0 y_1 y_3 a^q \gamma^{q^3} + y_0 y_1 y_3 a^{q^3} \gamma - y_0 y_1 y_3 a^{q^3} \gamma^q +y_0 y_1 y_4 y_5 a^{q^5+1}
\\&
+ y_0 y_1 y_4 y_5 \gamma - y_0 y_1 y_4 y_5 \gamma^q - y_0 y_1 y_4 a \gamma^{q^3} + y_0 y_1 y_4 a \gamma^{q^4} + y_0 y_1 y_4 a^{q^4} \gamma -y_0 y_1 y_4 a^{q^4} \gamma^q - 
\\&
y_0 y_1 y_5 a^q \gamma^{q^2} + y_0 y_1 \gamma^{q^3+1} - y_0 y_1 \gamma^{q^3+q} - y_0 y_2 y_3 y_4a^{q^2+1} + y_0 y_2 y_3 a^{q^2} \gamma - y_0 y_2 y_4 a^{q^4+q^2+1}
\\&
- y_0 y_2 y_5 a^{q^5+q^2+1} - y_0 y_2 y_5 a^{q^2} \gamma - y_0 y_2a^{q^2+1} \gamma^{q^4} - y_0 y_3 y_4 y_5 a^{q^5+1} - y_0 y_3 y_4 y_5 \gamma - y_0 y_3 y_4 a^{q^4+q^3+1}
\\&
-y_0 y_3 y_4 a \gamma^{q^2} + y_0 y_3 y_4 a \gamma^{q^3} - y_0 y_3 y_4 a \gamma^{q^4} - y_0 y_3 y_4 a^{q^4} \gamma - y_0 y_3 y_5 a^{q^5+q^3+1} - y_0 y_3 y_5 a^{q^3} \gamma 
\\&
-y_0 y_3 a^{q^3+1} \gamma^{q^4} + y_0 y_3 \gamma^{q^2+1} - y_0 y_3 \gamma^{q^3+1} - y_0 y_4 a^{q^4+1} \gamma^{q^2} - y_0 y_5 a^{q^5+1} \gamma^{q^2} - y_0 y_5 \gamma^{q^2+1}\\&
- y_0 a \gamma^{q^4+q^2} -y_1 y_2 y_3 y_4 y_5 a^{q^5} - y_1 y_2 y_3 y_4 a^{q^2+q} + y_1 y_2 y_3 y_4 a^{q^3+q^2} - y_1 y_2 y_3 y_4 a^{q^4+q^3} - y_1 y_2 y_3 y_4 \gamma
\\&
+ y_1 y_2 y_3 y_4 \gamma^q -y_1 y_2 y_3 y_4 \gamma^{q^2} + y_1 y_2 y_3 y_4 \gamma^{q^3} - y_1 y_2 y_3 y_4 \gamma^{q^4} - y_1 y_2 y_3 y_5 a^{q^5+q^3} - y_1 y_2 y_3 a^{q^3} \gamma^{q^4}
\\&
- y_1 y_2 y_4 y_5 a^{q^5+q^2} -y_1 y_2 y_4 a^{q^4+q^2+q} + y_1 y_2 y_4 a^{q^2} \gamma^{q^3} - y_1 y_2 y_4 a^{q^2} \gamma^{q^4} - y_1 y_2 y_4 a^{q^4} \gamma 
\\&
+ y_1 y_2 y_4 a^{q^4} \gamma^q - y_1 y_2 y_4 a^{q^4} \gamma^{q^2} -y_1 y_2 y_5 a^{q^5+q^2+q} - y_1 y_2 y_5 a^{q^5} \gamma + y_1 y_2 y_5 a^{q^5} \gamma^q
\\&
- y_1 y_2 y_5 a^{q^5} \gamma^{q^2} - y_1 y_2 a^{q^2+q} \gamma^{q^4} - y_1 y_2 \gamma^{q^4+1} + y_1 y_2 \gamma^{q^4+q} -y_1 y_2 \gamma^{q^4+q^2} - y_1 y_3 y_4 y_5 a^{q^5+q}
\\&
- y_1 y_3 y_4 a^{q^4+q^3+q} - y_1 y_3 y_4 a^q \gamma^{q^2} + y_1 y_3 y_4 a^q \gamma^{q^3} - y_1 y_3 y_4 a^q \gamma^{q^4} - y_1 y_3 y_4 a^{q^3} \gamma + y_1 y_3 y_4 a^{q^3} \gamma^q 
\\&
- y_1 y_3 y_5 a^{q^5+q^3+q} - y_1 y_3 a^{q^3+q} \gamma^{q^4} + y_1 y_4 y_5 a^{q^5} \gamma - y_1 y_4 y_5 a^{q^5} \gamma^q -y_1 y_4 a^{q^4+q} \gamma^{q^2} - y_1 y_4 \gamma^{q^3+1}
\\&
+ y_1 y_4 \gamma^{q^4+1} + y_1 y_4 \gamma^{q^3+q} - y_1 y_4 \gamma^{q^4+q} - y_1 y_5 a^{q^5+q} \gamma^{q^2} - y_1 a^q \gamma^{q^4+q^2} -y_2 y_3 y_4 a^{q^2} \gamma
\\&
- y_2 y_4 a^{q^4+q^2} \gamma- y_2 y_5 a^{q^5+q^2} \gamma - y_2 a^{q^2} \gamma^{q^4+1} - y_3 y_4 y_5 a^{q^5} \gamma - y_3 y_4 a^{q^4+q^3} \gamma \\&
- y_3 y_4 \gamma^{q^2+1} +y_3 y_4 \gamma^{q^3+1} -y_3 y_4 \gamma^{q^4+1} - y_3 y_5 a^{q^5+q^3} \gamma - y_3 a^{q^3} \gamma^{q^4+1}- y_4 a^{q^4} \gamma^{q^2+1}\\& - y_5 a^{q^5} \gamma^{q^2+1} - \gamma^{q^4+q^2+1};
\end{align*}

\begin{align*}
g_1({\bf y}) :=&- y_1 y_2 y_4 a^{q^4}+y_1 y_2 y_4 a^{q^2}  + y_1 y_2 y_5 a^q - y_1 y_2 y_5 a^{q^5}+ y_1 y_2 a^{q^2+q} - y_1 y_2\gamma^{q^4} - y_1 y_4 y_5 a^{q}\\
& + y_1 y_4 y_5 a^{q^5} - y_1 y_4 a^{q^4+q} + y_1 y_4 \gamma^{q^2} - y_1 y_4\gamma^{q^3} + y_1 y_4\gamma^{q^4} + y_1 a^q \gamma^{q^2} - y_1 a^q\gamma^{q^3} - y_2 y_4 y_5 a^{q^2}\\
& + y_2 y_4 y_5 a^{q^4}- y_2 y_5 a^{q^5+q^2} +y_2 y_5 \gamma^q - y_2 y_5 \gamma^{q^2} + y_2 y_5\gamma^{q^3} + y_2 a^{q^2} \gamma^q - y_2 a^{q^2} \gamma^{q^2}\\
&  + y_2 a^{q^2}\gamma^{q^3} - y_2 a^{q^2}\gamma^{q^4} + y_4 y_5 a^{q^5+q^4}- y_4 y_5 \gamma^q - y_4 a^{q^4} \gamma^q+ y_4 a^{q^4} \gamma^{q^2} - y_4 a^{q^4}\gamma^{q^3} + y_4 a^{q^4}\gamma^{q^4}\\
&- y_5 a^{q^5} \gamma^{q^2} + y_5 a^{q^5}\gamma^{q^3} + \gamma^{q^2+q} - \gamma^{q^3+q} - \gamma^{2q^2} + 2 \gamma^{q^3+q^2} - \gamma^{q^4+q^2} -\gamma^{2q^3} + \gamma^{q^4+q^3};
\end{align*}

\begin{align*}
g_2({\bf y}):=& +y_2 y_3 y_4 y_5 a^{q+1} - y_2 y_3 y_4 y_5 a^{q^5+1} - y_2 y_3 y_4 y_5 a^{q^2+q} + y_2 y_3 y_4 y_5 a^{q^3+q^2}\\
&- y_2 y_3 y_4 y_5 a^{q^4+q^3} + y_2 y_3 y_4 y_5 a^{q^5+q^4} + y_2 y_3 y_4 a^{q^3+q^2+1} - y_2 y_3 y_4 a^{q^4+q^3+1} + y_2 y_3 y_4 a \gamma^q\\
&- y_2 y_3 y_4 a \gamma^{q^2} + y_2 y_3 y_4 a\gamma^{q^3} - y_2 y_3 y_4 a\gamma^{q^4} - y_2 y_3 y_4 a^{q^2} \gamma + y_2 y_3 y_4 a^{q^4} \gamma - y_2 y_3 y_4 a^{q^4} \gamma^q\\
&+ y_2 y_3 y_4 a^{q^4} \gamma^{q^2} - y_2 y_3 y_4 a^{q^4}\gamma^{q^3} + y_2 y_3 y_4 a^{q^4}\gamma^{q^4} + y_2 y_3 y_5 a^{q^3+q+1} - y_2 y_3 y_5 a^{q^5+q^3+1}\\
&- y_2 y_3 y_5 a^{q^5+q^2+q} + y_2 y_3 y_5a^{q^5+q^3+q^2} - y_2 y_3 y_5 a^{q^3} \gamma^q + y_2 y_3 y_5 a^{q^3} \gamma^{q^2} - y_2 y_3 y_5 a^{q^3}\gamma^{q^3}\\
&+ y_2 y_3 y_5 a^{q^5} \gamma^q - y_2 y_3 y_5 a^{q^5} \gamma^{q^2} + y_2 y_3 y_5 a^{q^5}\gamma^{q^3} + y_2 y_3 a^{q^3+q^2+q+1} - y_2 y_3 a^{q^3+1}\gamma^{q^4}\\
&- y_2 y_3 a^{q^2+q} \gamma + y_2 y_3 a^{q^2+q} \gamma^q - y_2 y_3 a^{q^2+q} \gamma^{q^2} + y_2 y_3 a^{q^2+q}\gamma^{q^3} - y_2 y_3 a^{q^2+q}\gamma^{q^4}\\
&- y_2 y_3 a^{q^3+q^2}\gamma^q + y_2 y_3 a^{q^3+q^2}\gamma^{q^2} - y_2 y_3 a^{q^3+q^2}\gamma^{q^3} + y_2 y_3 a^{q^3+q^2}\gamma^{q^4} + y_2 y_3 \gamma^{q+1}\\
&- y_2 y_3 \gamma^{q^2+1}+ y_2 y_3 \gamma^{q^3+1} - y_2 y_3 \gamma^{2q} + 2 y_2 y_3 \gamma^{q^2+q} - 2 y_2 y_3 \gamma^{q^3+q} + y_2 y_3 \gamma^{q^4+q}\\
&- y_2 y_3 \gamma^{2q^2} + 2 y_2 y_3 \gamma^{q^3+q^2} - y_2 y_3 \gamma^{q^4+q^2} - y_2 y_3\gamma^{2q^3} + y_2 y_3\gamma^{q^4+q^3} + y_2 y_4 y_5 a^{q^4+q+1}\\
&- y_2 y_4 y_5 a^{q^5+q^2+1} - y_2 y_4 y_5 a^{q^4+q^2+q} + y_2 y_4 y_5 a^{q^5+q^4+q^2} + y_2 y_4 y_5 a^{q^2} \gamma - y_2 y_4 y_5 a^{q^2} \gamma^q\\
&+ y_2 y_4 y_5 a^{q^2} \gamma^{q^2} - y_2 y_4 y_5 a^{q^4} \gamma + y_2 y_4 y_5 a^{q^4} \gamma^q - y_2 y_4 y_5 a^{q^4} \gamma^{q^2} + y_2 y_4 a^{q^2+1}\gamma^{q^3}\\
&- y_2 y_4 a^{q^2+1}\gamma^{q^4} + y_2 y_4 a^{q^4+1} \gamma^q - y_2 y_4 a^{q^4+1} \gamma^{q^2} - y_2 y_4 a^{q^4+q^2}\gamma^q + y_2 y_4 a^{q^4+q^2}\gamma^{q^2}\\
&- y_2 y_4 a^{q^4+q^2}\gamma^{q^3} + y_2 y_4 a^{q^4+q^2}\gamma^{q^4} - y_2 y_5 a^{q^5+q^2+q+1} + y_2 y_5 a^{q+1}\gamma^{q^3} + y_2 y_5 a^{q^5+1} \gamma^q\\
&- y_2 y_5 a^{q^5+1} \gamma^{q^2} + y_2 y_5 a^{q^2+q} \gamma - y_2 y_5 a^{q^2+q} \gamma^q + y_2 y_5 a^{q^2+q} \gamma^{q^2} - y_2 y_5 a^{q^2+q}\gamma^{q^3}\\
&+ y_2 y_5 a^{q^5+q^2}\gamma^{q^3} - y_2 y_5 \gamma^{q+1}+ y_2 y_5 \gamma^{q^2+1}- y_2 y_5 \gamma^{q^3+1} + y_2 y_5 \gamma^{2q}- 2 y_2 y_5 \gamma^{q^2+q}\\
&+ y_2 y_5 \gamma^{q^3+q} + y_2 y_5 \gamma^{2q^2} - y_2 y_5 \gamma^{q^3+q^2} + y_2 a^{q^2+q+1}\gamma^{q^3} - y_2 a^{q^2+q+1}\gamma^{q^4} + y_2 a \gamma^{q^4+q}\\
&- y_2 a \gamma^{q^4+q^2} - y_2 a^{q^2} \gamma^{q^3+q} + y_2 a^{q^2} \gamma^{q^3+q^2} - y_2 a^{q^2}\gamma^{2q^3} + y_2 a^{q^2}\gamma^{q^4+q^3} - y_3 y_4 y_5 a^{q^4+q^3+q}\\
&+ y_3 y_4 y_5 a^{q^5+q^4+q} + y_3 y_4 y_5 a^q \gamma - y_3 y_4 y_5 a^q \gamma^q + y_3 y_4 y_5 a^{q^3} \gamma^q - y_3 y_4 y_5 a^{q^5} \gamma\\
&- y_3 y_4 a^{q^4+q^3+q+1} + y_3 y_4 a^{q^3+1}\gamma^q + y_3 y_4 a^{q^4+q} \gamma - y_3 y_4 a^{q^4+q} \gamma^q + y_3 y_4 a^{q^4+q} \gamma^{q^2}\\
&- y_3 y_4 a^{q^4+q}\gamma^{q^3} + y_3 y_4 a^{q^4+q}\gamma^{q^4} - y_3 y_4 \gamma^{q^2+1}+ y_3 y_4 \gamma^{q^3+1} - y_3 y_4 \gamma^{q^4+1} + y_3 y_5 a^{q^3+q} \gamma\\
&- y_3 y_5 a^{q^3+q} \gamma^q + y_3 y_5 a^{q^3+q} \gamma^{q^2} - y_3 y_5 a^{q^3+q}\gamma^{q^3} - y_3 y_5 a^{q^5+q} \gamma^{q^2} + y_3 y_5 a^{q^5+q}\gamma^{q^3}\\
&- y_3 y_5 a^{q^5+q^3} \gamma + y_3 y_5 a^{q^5+q^3} \gamma^q + y_3 a^{q^3+q+1}\gamma^{q^2} - y_3 a^{q^3+q+1}\gamma^{q^3} - y_3 a^q \gamma^{q^2+1}\\
&+ y_3 a^q \gamma^{q^3+1} + y_3 a^q \gamma^{q^2+q} - y_3 a^q \gamma^{q^3+q} - y_3 a^q \gamma^{2q^2} + 2 y_3 a^q \gamma^{q^3+q^2} - y_3 a^q \gamma^{q^4+q^2}\\
&- y_3 a^q\gamma^{2q^3} + y_3 a^q\gamma^{q^4+q^3} + y_3 a^{q^3} \gamma^{q+1}- y_3 a^{q^3} \gamma^{q^2+1}+ y_3 a^{q^3} \gamma^{q^3+1} - y_3 a^{q^3} \gamma^{q^4+1}- y_3 a^{q^3} \gamma^{2q}\\
&+ y_3 a^{q^3} \gamma^{q^2+q} - y_3 a^{q^3} \gamma^{q^3+q} + y_3 a^{q^3} \gamma^{q^4+q} + y_4 y_5 a^{q^5+q^4+q+1} - y_4 y_5 a^{q^5+1} \gamma^q\\
&- y_4 y_5 a^{q^4+q} \gamma^{q^2} -y_4 y_5 a^{q^5+q^4} \gamma + y_4 y_5 a^{q^5+q^4} \gamma^q + y_4 y_5 \gamma^{q+1}- y_4 y_5 \gamma^{2q} + y_4 y_5 \gamma^{q^2+q}\\
&- y_4 a^{q^4+q+1}\gamma^{q^3} + y_4 a^{q^4+q+1}\gamma^{q^4} + y_4 a \gamma^{q^3+q} - y_4 a \gamma^{q^4+q} + y_4 a^{q^4} \gamma^{q+1}- y_4 a^{q^4} \gamma^{q^2+1}\\
&+ y_4 a^{q^4} \gamma^{q^3+1} - y_4 a^{q^4} \gamma^{q^4+1} - y_4 a^{q^4} \gamma^{2q} + y_4 a^{q^4} \gamma^{q^2+q} -y_4 a^{q^4} \gamma^{q^3+q} + y_4 a^{q^4} \gamma^{q^4+q}\displaybreak\\
& - y_5 a^{q^5+q+1} \gamma^{q^2}+ y_5 a^{q^5+q+1}\gamma^{q^3} + y_5 a^q \gamma^{q^2+1}- y_5 a^q \gamma^{q^2+q} + y_5 a^q \gamma^{2q^2} - y_5 a^q \gamma^{q^3+q^2}\\& - y_5 a^{q^5} \gamma^{q^3+1}+ y_5 a^{q^5} \gamma^{q^3+q} + a^{q+1} \gamma^{q^3+q^2} - a^{q+1} \gamma^{q^4+q^2} - a^{q+1} \gamma^{2q^3} + a^{q+1} \gamma^{q^4+q^3} + \gamma^{q^3+q+1}\\
& - \gamma^{q^3+q^2+1}+ \gamma^{2q^3+1} - \gamma^{q^4+q^3+1} - \gamma^{q^3+2q} + \gamma^{q^3+q^2+q} - \gamma^{2q^3+1} + \gamma^{q^4+q^3+q};
\end{align*}

\begin{align*}
g_3({\bf y}) :=& +y_0 y_2 y_4 y_5^2 a^{q+1} - y_0 y_2 y_4 y_5^2 a^{q^5+1} - y_0 y_2 y_4 y_5^2 a^{q^2+q} + y_0 y_2 y_4 y_5^2 a^{q^3+q^2}- y_0 y_2 y_4 y_5^2 a^{q^4+q^3}\\
& + y_0 y_2 y_4 y_5^2 a^{q^5+q^4} + y_0 y_2 y_4 y_5 a^{q^4+q+1} + y_0 y_2 y_4 y_5 a^{q^3+q^2+1}- y_0 y_2 y_4 y_5 a^{q^5+q^2+1} - y_0 y_2 y_4 y_5 a^{q^4+q^3+1}\\
& + y_0 y_2 y_4 y_5 a \gamma^q - y_0 y_2 y_4 y_5 a \gamma^{q^2}+ y_0 y_2 y_4 y_5 a\gamma^{q^3} - y_0 y_2 y_4 y_5 a\gamma^{q^4} - y_0 y_2 y_4 y_5 a^{q^4+q^2+q}\\
& + y_0 y_2 y_4 y_5 a^{q^5+q^4+q^2}- y_0 y_2 y_4 y_5 a^{q^2} \gamma^q + y_0 y_2 y_4 y_5 a^{q^2} \gamma^{q^2} - y_0 y_2 y_4 y_5 a^{q^4}\gamma^{q^3} + y_0 y_2 y_4 y_5 a^{q^4}\gamma^{q^4} \\
&+ y_0 y_2 y_4 a^{q^2+1}\gamma^{q^3} - y_0 y_2 y_4 a^{q^2+1}\gamma^{q^4} + y_0 y_2 y_4 a^{q^4+1} \gamma^q - y_0 y_2 y_4 a^{q^4+1} \gamma^{q^2}- y_0 y_2 y_4 a^{q^4+q^2}\gamma^q\\
& + y_0 y_2 y_4 a^{q^4+q^2}\gamma^{q^2} - y_0 y_2 y_4 a^{q^4+q^2}\gamma^{q^3} + y_0 y_2 y_4 a^{q^4+q^2}\gamma^{q^4}+ y_0 y_2 y_5^2 a^{q^3+q+1}- y_0 y_2 y_5^2 a^{q^5+q^3+1}\\
& - y_0 y_2 y_5^2 a^{q^5+q^2+q} + y_0 y_2 y_5^2a^{q^5+q^3+q^2}- y_0 y_2 y_5^2 a^{q^3} \gamma^q + y_0 y_2 y_5^2 a^{q^3} \gamma^{q^2} - y_0 y_2 y_5^2 a^{q^3}\gamma^{q^3}\\
& + y_0 y_2 y_5^2 a^{q^5} \gamma^q - y_0 y_2 y_5^2 a^{q^5} \gamma^{q^2}+ y_0 y_2 y_5^2 a^{q^5}\gamma^{q^3} + y_0 y_2 y_5 a^{q^3+q^2+q+1} - y_0 y_2 y_5 a^{q^5+q^2+q+1}\\
& + y_0 y_2 y_5 a^{q+1}\gamma^{q^3}- y_0 y_2 y_5 a^{q^3+1}\gamma^{q^4} + y_0 y_2 y_5 a^{q^5+1} \gamma^q - y_0 y_2 y_5 a^{q^5+1} \gamma^{q^2} - y_0 y_2 y_5 a^{q^2+q}\gamma^{q^4}\\
&- y_0 y_2 y_5 a^{q^3+q^2}\gamma^q + y_0 y_2 y_5 a^{q^3+q^2}\gamma^{q^2} - y_0 y_2 y_5 a^{q^3+q^2}\gamma^{q^3} + y_0 y_2 y_5 a^{q^3+q^2}\gamma^{q^4}+ y_0 y_2 y_5 a^{q^5+q^2}\gamma^{q^3}\\
& - y_0 y_2 y_5 \gamma^{q^3+q} + y_0 y_2 y_5 \gamma^{q^4+q} + y_0 y_2 y_5 \gamma^{q^3+q^2} - y_0 y_2 y_5 \gamma^{q^4+q^2}- y_0 y_2 y_5\gamma^{2q^3} + y_0 y_2 y_5\gamma^{q^4+q^3}\\
& + y_0 y_2 a^{q^2+q+1}\gamma^{q^3} - y_0 y_2 a^{q^2+q+1}\gamma^{q^4} + y_0 y_2 a \gamma^{q^4+q}- y_0 y_2 a \gamma^{q^4+q^2} - y_0 y_2 a^{q^2} \gamma^{q^3+q}\\
& + y_0 y_2 a^{q^2} \gamma^{q^3+q^2} - y_0 y_2 a^{q^2}\gamma^{2q^3}+ y_0 y_2 a^{q^2}\gamma^{q^4+q^3}+ y_0 y_4 y_5^2 a^{q^5+q+1} - y_0 y_4 y_5^2 a^{2q^5+1} - y_0 y_4 y_5^2 a^q \gamma^{q^2}\\
& - y_0 y_4 y_5^2 a^{q^5+q^4+q^3}+ y_0 y_4 y_5^2 a^{q^3} \gamma^q + y_0 y_4 y_5^2 a^{2q^5+q^4} - y_0 y_4 y_5^2 a^{q^5} \gamma^q + y_0 y_4 y_5^2 a^{q^5} \gamma^{q^2}\\
&+ y_0 y_4 y_5 a^{q^5+q^4+q+1} - y_0 y_4 y_5 a^{q+1} \gamma^{q^3} + y_0 y_4 y_5 a^{q+1}\gamma^{q^4} - y_0 y_4 y_5 a^{q^5+q^4+q^3+1}+ y_0 y_4 y_5 a^{q^3+1}\gamma^q\\
& - y_0 y_4 y_5 a^{q^5+1} \gamma^{q^2} + 2 y_0 y_4 y_5 a^{q^5+1}\gamma^{q^3} - 2 y_0 y_4 y_5 a^{q^5+1}\gamma^{q^4}- y_0 y_4 y_5 a^{q^4+q} \gamma^{q^2} + y_0 y_4 y_5 a^{q^4+q^3} \gamma^q\\
& - y_0 y_4 y_5 a^{q^4+q^3} \gamma^{q^2} + y_0 y_4 y_5 a^{q^4+q^3}\gamma^{q^3}- y_0 y_4 y_5 a^{q^4+q^3}\gamma^{q^4} - y_0 y_4 y_5 a^{q^5+q^4} \gamma^q + 2 y_0 y_4 y_5 a^{q^5+q^4} \gamma^{q^2}\\
& - 2 y_0 y_4 y_5 a^{q^5+q^4}\gamma^{q^3}+ 2 y_0 y_4 y_5 a^{q^5+q^4}\gamma^{q^4} - y_0 y_4 y_5 \gamma^{q^2+q} + y_0 y_4 y_5 \gamma^{q^3+q} - y_0 y_4 y_5 \gamma^{q^4+q} + y_0 y_4 y_5 \gamma^{2q^2}\\
&-y_0 y_4 y_5 \gamma^{q^3+q^2} + y_0 y_4 y_5 \gamma^{q^4+q^2} - y_0 y_4 a^{q^4+q+1}\gamma^{q^3} + y_0 y_4 a^{q^4+q+1}\gamma^{q^4} + y_0 y_4 a^{q^5+q^4+q^3+1} \gamma^q\\
&- y_0 y_4 a^{q^4} \gamma^{q^4+q} + y_0 y_4 a^{q^4} \gamma^{2q^2} - 2 y_0 y_4 a^{q^4} \gamma^{q^3+q^2} + 2 y_0 y_4 a^{q^4} \gamma^{q^4+q^2} + y_0 y_4 a^{q^4}\gamma^{2q^3}- 2 y_0 y_4 a^{q^4}\gamma^{q^4+q^3}\\
& + y_0 y_4 a^{q^4}\gamma^{2q^4} + y_0 y_5^2 a^{q^5+q^3+q+1} -y_0 y_5^2 a^{2q^5+q^3+1} - y_0 y_5^2a^{q^5+q} \gamma^{q^2}+ y_0 y_5^2 a^{q^5+q^3} \gamma^{q^2} - y_0 y_5^2 a^{q^5+q^3}\gamma^{q^3}\\
& + y_0 y_5^2 a^{2q^5}\gamma^{q^3} + y_0 y_5 a^{q^3+q+1}\gamma^{q^2} - y_0 y_5 a^{q^3+q+1}\gamma^{q^3}+ y_0 y_5 a^{q^3+q+1}\gamma^{q^4} - y_0 y_5 a^{q^5+q+1} \gamma^{q^2}\\
& + y_0 y_5 a^{q^5+q+1}\gamma^{q^3} + y_0 y_5 a^{2q^5+q^3+1} \gamma^q- y_0 y_5 a^{2q^5+q^3+1} \gamma^{q^2} + y_0 y_5 a^{2q^5+q^3+1}\gamma^{q^3} - 2 y_0 y_5 a^{2q^5+q^3+1}\gamma^{q^4}\\
& - y_0 y_5 a^q \gamma^{q^4+q^2} - y_0 y_5 a^{q^3} \gamma^{q^2+q}+ y_0 y_5 a^{q^3} \gamma^{q^3+q} + y_0 y_5 a^{q^3} \gamma^{2q^2} - 2 y_0 y_5 a^{q^3} \gamma^{q^3+q^2}\displaybreak\\
& + y_0 y_5 a^{q^3} \gamma^{q^4+q^2} + y_0 y_5 a^{q^3}\gamma^{2q^3}- y_0 y_5 a^{q^3}\gamma^{q^4+q^3} - y_0 y_5 a^{q^5} \gamma^{q^3+q} + 2 y_0 y_5 a^{q^5} \gamma^{q^3+q^2} - 2 y_0 y_5 a^{q^5}\gamma^{2q^3}\\
& + 2 y_0 y_5 a^{q^5} \gamma^{q^4+q^3}+ y_0 a^{q+1} \gamma^{q^3+q^2} - y_0 a^{q+1} \gamma^{q^4+q^2} - y_0 a^{q+1} \gamma^{2q^3} + y_0 a^{q+1} \gamma^{q^4+q^3} + y_0 a^{q^3+1}\gamma^{q^4+q}\\
& - y_0 a^{q^3+1} \gamma^{q^4+q^2}+ y_0 a^{q^3+1}\gamma^{q^4+q^3} - y_0 a^{q^3+1}\gamma^{2q^4} - y_0 \gamma^{q^3+q^2+q} + y_0 \gamma^{2q^3+q} - y_0 \gamma^{q^4+q^3+q} + y_0 \gamma^{q^3+2q^2}\\
& - 2 y_0 \gamma^{2q^3+q^2}+ 2 y_0 \gamma^{q^4+q^3+q^2} + y_0\gamma^{3q^3} -2 y_0\gamma^{q^4+2q^3} + y_0\gamma^{2q^4+q^3} - y_2 y_4 y_5^2 a^{q^5+q^2+q} - y_2 y_4 y_5^2 a^{q^4+q^3+q}\\
&+ y_2 y_4 y_5^2 a^{q^5+q^4+q} + y_2 y_4 y_5^2 a^q \gamma- y_2 y_4 y_5^2 a^q \gamma^q + y_2 y_4 y_5^2 a^q \gamma^{q^2}+ y_2 y_4 y_5^2a^{q^5+q^3+q^2} - y_2 y_4 y_5^2 a^{q^5} \gamma\\
& + y_2 y_4 y_5^2 a^{q^5} \gamma^q - y_2 y_4 y_5^2 a^{q^5} \gamma^{q^2}- y_2 y_4 y_5 a^{q^4+q^3+q^2+q} + y_2 y_4 y_5 a^{q^2+q}\gamma^{q^3} - y_2 y_4 y_5 a^{q^2+q}\gamma^{q^4}\\
& + y_2 y_4 y_5 a^{q^4+q} \gamma- y_2 y_4 y_5 a^{q^4+q} \gamma^q + y_2 y_4 y_5 a^{q^4+q} \gamma^{q^2} - y_2 y_4 y_5 a^{q^4+q}\gamma^{q^3} + y_2 y_4 y_5 a^{q^4+q}\gamma^{q^4}\\
&+ y_2 y_4 y_5 a^{q^5+q^4+q^3+q^2} + y_2 y_4 y_5 a^{q^3+q^2} \gamma - y_2 y_4 y_5 a^{q^3+q^2} \gamma^q + y_2 y_4 y_5 a^{q^3+q^2} \gamma^{q^2}- y_2 y_4 y_5 a^{q^3+q^2}\gamma^{q^3}\\
& + y_2 y_4 y_5 a^{q^3+q^2}\gamma^{q^4} - y_2 y_4 y_5 a^{q^5+q^2} \gamma - y_2 y_4 y_5 a^{q^4+q^3} \gamma+ y_2 y_4 y_5 a^{q^5+q^4} \gamma^q - y_2 y_4 y_5 a^{q^5+q^4} \gamma^{q^2}\\
& + y_2 y_4 y_5 \gamma^{q+1}- y_2 y_4 y_5 \gamma^{q^2+1}+ y_2 y_4 y_5 \gamma^{q^3+1}- y_2 y_4 y_5 \gamma^{q^4+1} - y_2 y_4 y_5 \gamma^{2q} + 2 y_2 y_4 y_5 \gamma^{q^2+q}\\
& - y_2 y_4 y_5 \gamma^{q^3+q} + y_2 y_4 y_5 \gamma^{q^4+q}- y_2 y_4 y_5 \gamma^{2q^2} + y_2 y_4 y_5 \gamma^{q^3+q^2} - y_2 y_4 y_5 \gamma^{q^4+q^2}- y_2 y_4 a^{q^4+q^3+q^2}\gamma^q \\
& + y_2 y_4 a^{q^4+q^3+q^2}\gamma^{q^2}- y_2 y_4 a^{q^4+q^3+q^2}\gamma^{q^3} + y_2 y_4 a^{q^4+q^3+q^2}\gamma^{q^4} + y_2 y_4 a^{q^2} \gamma^{q^3+1} - y_2 y_4 a^{q^2} \gamma^{q^4+1}\\
& + y_2 y_4 a^{q^4} \gamma^{q+1}- y_2 y_4 a^{q^4} \gamma^{q^2+1}- y_2 y_4 a^{q^4} \gamma^{2q} + 2 y_2 y_4 a^{q^4} \gamma^{q^2+q} - y_2 y_4 a^{q^4} \gamma^{q^3+q} + y_2 y_4 a^{q^4} \gamma^{q^4+q}\\
&- y_2 y_4 a^{q^4}  \gamma^{2q^2} + y_2 y_4 a^{q^4} \gamma^{q^3+q^2} - y_2 y_4 a^{q^4} \gamma^{q^4+q^2} - y_2 y_5^2 a^{2q^5+q^2+q} + y_2 y_5^2 a^{q^3+q} \gamma- y_2 y_5^2 a^{q^3+q} \gamma^q\\
& + y_2 y_5^2 a^{q^3+q} \gamma^{q^2} - y_2 y_5^2 a^{q^3+q}\gamma^{q^3} + y_2 y_5^2a^{q^5+q}\gamma^{q^3} + y_2 y_5^2 a^{2q^5+q^3+q^2}- y_2 y_5^2 a^{q^5+q^3} \gamma\\
& + y_2 y_5^2 a^{2q^5} \gamma^q - y_2 y_5^2a^{2q^5} \gamma^{q^2} + y_2 y_5a^{q^3+q^2+q} \gamma - y_2 y_5 a^{q^3+q^2+q} \gamma^q+ y_2 y_5 a^{q^3+q^2+q} \gamma^{q^2}\\
& - y_2 y_5 a^{q^3+q^2+q}\gamma^{q^3} - y_2 y_5 a^{q^5+q^2+q} \gamma + y_2 y_5 a^{q^5+q^2+q} \gamma^q- y_2 y_5 a^{q^5+q^2+q} \gamma^{q^2} + 2 y_2 y_5 a^{q^5+q^2+q}\gamma^{q^3}\\
& - 2 y_2 y_5 a^{q^5+q^2+q}\gamma^{q^4} + y_2 y_5 a^q \gamma^{q^3+1}- y_2 y_5 a^q \gamma^{q^3+q} + y_2 y_5 a^q \gamma^{q^3+q^2} - y_2 y_5 a^q\gamma^{2q^3}\\
& + y_2 y_5 a^q\gamma^{q^4+q^3} - y_2 y_5a^{q^5+q^3+q^2} \gamma^q+ y_2 y_5a^{q^5+q^3+q^2} \gamma^{q^2} - y_2 y_5a^{q^5+q^3+q^2}\gamma^{q^3} + 2 y_2 y_5a^{q^5+q^3+q^2}\gamma^{q^4}\\
& - y_2 y_5 a^{q^3} \gamma^{q^4+1} + y_2 y_5 a^{q^5}\gamma^{q+1}- y_2 y_5 a^{q^5} \gamma^{q^2+1}- y_2 y_5 a^{q^5} \gamma^{2q} + 2 y_2 y_5 a^{q^5} \gamma^{q^2+q} - y_2 y_5 a^{q^5} \gamma^{q^3+q}\\
& + 2 y_2 y_5 a^{q^5} \gamma^{q^4+q}- y_2 y_5 a^{q^5} \gamma^{2q^2} + y_2 y_5 a^{q^5} \gamma^{q^3+q^2} - 2 y_2 y_5 a^{q^5} \gamma^{q^4+q^2} + y_2 a^{q^2+q} \gamma^{q^3+1} - y_2 a^{q^2+q} \gamma^{q^4+1}\\
&- y_2 a^{q^2+q} \gamma^{q^3+q} + y_2 a^{q^2+q} \gamma^{q^4+q} + y_2 a^{q^2+q} \gamma^{q^3+q^2} - y_2 a^{q^2+q} \gamma^{q^4+q^2} - y_2 a^{q^2+q}\gamma^{2q^3}\\
&+ 2 y_2 a^{q^2+q}\gamma^{q^4+q^3} - y_2 a^{q^2+q}\gamma^{2q^4} - y_2 a^{q^3+q^2} \gamma^{q^4+q} + y_2 a^{q^3+q^2} \gamma^{q^4+q^2} - y_2 a^{q^3+q^2}\gamma^{q^4+q^3} + y_2 a^{q^3+q^2}\gamma^{2q^4}\\
&+ y_2 \gamma^{q^4+q+1} - y_2\gamma^{q^4+q^2+1} - y_2 \gamma^{q^4+2q} + 2 y_2 \gamma^{q^4+q^2+q} - y_2 \gamma^{q^4+q^3+q} + y_2 \gamma^{2q^4+q} - y_2 \gamma^{q^4+2q^2}\\
&+ y_2 \gamma^{q^4+q^3+q^2} - y_2 \gamma^{2q^4+q^2} - y_4 y_5^2 a^{q^5+q^4+q^3+q} + y_4 y_5^2 a^{2q^5+q^4+q} + y_4 y_5^2a^{q^5+q} \gamma\\
&- y_4 y_5^2a^{q^5+q} \gamma^q + y_4 y_5^2 a^{q^5+q^3} \gamma^q - y_4 y_5^2 a^{2q^5} \gamma - y_4 y_5 a^{q^4+q^3+q} \gamma^{q^2} + y_4 y_5 a^{q^4+q^3+q}\gamma^{q^3}\\
&- y_4 y_5 a^{q^4+q^3+q}\gamma^{q^4} + y_4 y_5 a^{q^5+q^4+q} \gamma - y_4 y_5 a^{q^5+q^4+q} \gamma^q + y_4 y_5 a^{q^5+q^4+q} \gamma^{q^2}- 2 y_4 y_5 a^{q^5+q^4+q}\gamma^{q^3}\\
& + 2 y_4 y_5 a^{q^5+q^4+q}\gamma^{q^4} - y_4 y_5 a^q \gamma^{q^3+1} + y_4 y_5 a^q \gamma^{q^4+1} + y_4 y_5 a^q \gamma^{q^3+q} - y_4 y_5 a^q \gamma^{q^4+q} - y_4 y_5 a^{q^5+q^4+q^3} \gamma\\
& + y_4 y_5 a^{q^5+q^4+q^3} \gamma^q + y_4 y_5 a^{q^3} \gamma^{q+1}- y_4 y_5 a^{q^3} \gamma^{2q}+ y_4 y_5 a^{q^3} \gamma^{q^2+q} - y_4 y_5 a^{q^3} \gamma^{q^3+q}\\
& + y_4 y_5 a^{q^3} \gamma^{q^4+q} - y_4 y_5 a^{q^5} \gamma^{q^2+1} + 2 y_4 y_5 a^{q^5} \gamma^{q^3+1}- 2 y_4 y_5 a^{q^5} \gamma^{q^4+1} - y_4 a^{q^4+q} \gamma^{q^3+1} + y_4 a^{q^4+q} \gamma^{q^4+1}\displaybreak\\
& + y_4 a^{q^4+q} \gamma^{q^3+q} - y_4 a^{q^4+q} \gamma^{q^4+q}- y_4 a^{q^4+q} \gamma^{q^3+q^2} + y_4 a^{q^4+q} \gamma^{q^4+q^2} + y_4 a^{q^4+q}\gamma^{2q^3} - 2 y_4 a^{q^4+q}\gamma^{q^4+q^3}\\
& + y_4 a^{q^4+q}\gamma^{2q^4} + y_4 a^{q^4+q^3} \gamma^{q+1}- y_4 a^{q^4+q^3} \gamma^{q^2+1}+ y_4 a^{q^4+q^3} \gamma^{q^3+1} - y_4 a^{q^4+q^3} \gamma^{q^4+1} - y_4 a^{q^4+q^3} \gamma^{2q}\\
& + y_4 a^{q^4+q^3} \gamma^{q^2+q} - y_4 a^{q^4+q^3} \gamma^{q^3+q}+ y_4 a^{q^4+q^3} \gamma^{q^4+q} + y_4 \gamma^{q^3+q^2+1} - y_4 \gamma^{q^4+q^2+1} - y_4 \gamma^{2q^3+1}\\
& + 2 y_4 \gamma^{q^4+q^3+1} - y_4 \gamma^{2q^4+1} + y_5^2 a^{q^5+q^3+q} \gamma- y_5^2a^{q^5+q^3+q} \gamma^{q} + y_5^2 a^{q^5+q^3+q} \gamma^{q^2} - y_5^2 a^{q^5+q^3+q}\gamma^{q^3}\\
& - y_5^2a^{2q^5+q} \gamma^{q^2} + y_5^2a^{2q^5+q}\gamma^{q^3}- y_5^2a^{2q^5+q^3} \gamma + y_5^2a^{2q^5+q^3} \gamma^{q} + y_5 a^{q^3+q} \gamma^{q^2+1}- y_5 a^{q^3+q} \gamma^{q^3+1}\\
& + y_5 a^{q^3+q} \gamma^{q^4+1}- y_5 a^{q^3+q} \gamma^{q^2+q} + y_5 a^{q^3+q} \gamma^{q^3+q} - y_5 a^{q^3+q} \gamma^{q^4+q} + y_5 a^{q^3+q} \gamma^{2q^2} - 2 y_5 a^{q^3+q} \gamma^{q^3+q^2}\\
&+ y_5 a^{q^3+q} \gamma^{q^4+q^2} + y_5 a^{q^3+q}\gamma^{2q^3} - y_5 a^{q^3+q}\gamma^{q^4+q^3} - y_5 a^{q^5+q} \gamma^{q^2+1} + y_5 a^{q^5+q} \gamma^{q^3+1} + y_5 a^{q^5+q} \gamma^{q^2+q}\\
&- y_5 a^{q^4+q} \gamma^{q^3+q} - y_5 a^{q^4+q} \gamma^{2q^2} + 3 y_5 a^{q^5+q} \gamma^{q^3+q^2} - 2 y_5 a^{q^5+q} \gamma^{q^4+q^2} - 2 y_5 a^{q^5+q}\gamma^{2q^3}\\
&+ 2 y_5 a^{q^5+q} \gamma^{q^4+q^3} + y_5 a^{q^5+q^3} \gamma^{q+1} - y_5 a^{q^5+q^3} \gamma^{q^2+1} + y_5a^{q^5+q^3} \gamma^{q^3+1} - 2 y_5a^{q^5+q^3} \gamma^{q^4+1}\\
&- y_5a^{q^5+q^3} \gamma^{2q} + y_5 a^{q^5+q^3} \gamma^{q^2+q} - y_5 a^{q^5+q^3} \gamma^{q^3+q} + 2 y_5a^{q^5+q^3} \gamma^{q^4+q} + a^q \gamma^{q^3+q^2+1} - a^q \gamma^{q^4+q^+12}\\
&- a^q \gamma^{2q^3+1} + a^q \gamma^{q^4+q^3+1} - a^q \gamma^{q^3+q^2+q} + a^q \gamma^{q^4+q^2+q} + a^q \gamma^{2q^3+q} - a^q \gamma^{q^4+q^3+q} + a^q \gamma^{q^3+2q^2}\\
&- a^q \gamma^{q^4+2q^2} - 2 a^q \gamma^{2q^3+q^2} + 3 a^q \gamma^{q^3+q^2}\gamma^{q^4} - a^q \gamma^{2q^4+q^2} + a^q\gamma^{3q^3} - 2 a^q\gamma^{q^4+2q^3} + a^q\gamma^{2q^4+q^3} + a^{q^3} \gamma^{q^4+q+1}\\
&- a^{q^3} \gamma^{q^4+q^2+1} + a^{q^3} \gamma^{q^4+q^3+1} - a^{q^3} \gamma^{2q^4+1} - a^{q^3} \gamma^{q^4+2q} + a^{q^3} \gamma^{q^4+q^2+q} - a^{q^3} \gamma^{q^4+q^3+q} + a^{q^3} \gamma^{2q^4+q};
\end{align*}
\begin{align*}
r_1(\mathbf{y})=&+y_4y_5a^{q^5+1} + y_4y_5a^{q^4+q^3} - y_4y_5a^{q^5+q^4} - y_4y_5\gamma + y_4y_5\gamma^{q} - y_4y_5\gamma^{q^2} + y_4a^{q^4+q^3+1} - y_4a\gamma^{q^3}\\& + y_4a\gamma^{q^4} 
            - y_4a^{q^4}\gamma + y_4a^{q^4}\gamma^{q} - y_4a^{q^4}\gamma^{q^2} + y_4a^{q^4}\gamma^{q^3} - y_4a^{q^4}\gamma^{q^4} + y_5a^{q^5+q^3+1} - y_5a^{q^3}\gamma\\& + y_5a^{q^3}\gamma^{q} - y_5a^{q^3}\gamma^{q^2} + 
            y_5a^{q^3}\gamma^{q^3} - y_5a^{q^5}\gamma^{q^3} + a^{q^3+1}\gamma^{q^4} - \gamma^{q^3+1} + \gamma^{q^3+q} - \gamma^{q^3+q^2} + \gamma^{2q^3} - \gamma^{q^4+q^3};
\end{align*}

\begin{align*}
r_2(\mathbf{y})=&+y_2y_4y_5^2a^{q+1} - y_2y_4y_5^2a^{q^5+1} - y_2y_4y_5^2a^{q^2+q} + y_2y_4y_5^2a^{q^3+q^2} - y_2y_4y_5^2a^{q^4+q^3} + 
            y_2y_4y_5^2a^{q^5+q^4}\\
            &+ y_2y_4y_5a^{q^4+q} + y_2y_4y_5a^{q^3+q^2+1} - y_2y_4y_5a^{q^5+q^2+1} - y_2y_4y_5a^{q^4+q^3+1} + 
            y_2y_4y_5a\gamma^q - y_2y_4y_5a\gamma^{q^2}\\
            &+ y_2y_4y_5a\gamma^{q^3} - y_2y_4y_5a\gamma^{q^4} - y_2y_4y_5a^{q^4+q^2+q} + y_2y_4y_5a^{q^5+q^4+q^2} 
            - y_2y_4y_5a^{q^2}\gamma^q + y_2y_4y_5a^{q^2}\gamma^{q^2}\\
            &- y_2y_4y_5a^{q^4}\gamma^{q^3} + y_2y_4y_5a^{q^4}\gamma^{q^4} + y_2y_4a^{q^2+1}\gamma^{q^3} - y_2y_4a^{q^2+1}\gamma^{q^4} + 
            y_2y_4a^{q^4+1}\gamma^q - y_2y_4a^{q^4+1}\gamma^{q^2}\\
            &- y_2y_4a^{q^4+q^2}\gamma^q + y_2y_4a^{q^4+q^2}\gamma^{q^2} - y_2y_4a^{q^4+q^2}\gamma^{q^3} + y_2y_4a^{q^4+q^2}\gamma^{q^4} + 
            y_2y_5^2a^{q^3+q+1} - y_2y_5^2a^{q^5+q^3+1}\\
            &- y_2y_5^2a^{q^5+q^2+q} + y_2y_5^2a^{q^5+q^3+q^2} - y_2y_5^2a^{q^3}\gamma^q + y_2y_5^2a^{q^3}\gamma^{q^2} 
            - y_2y_5^2a^{q^3}\gamma^{q^3} + y_2y_5^2a^{q^5}\gamma^q - y_2y_5^2a^{q^5}\gamma^{q^2} \\
            &+ y_2y_5^2a^{q^5}\gamma^{q^3} + y_2y_5a^{q^3+q^2+q+1} - y_2y_5a^{q^5+q^2+q+1} + 
            y_2y_5a^{q+1}\gamma^{q^3} - y_2y_5a^{q^3+1}\gamma^{q^4} + y_2y_5a^{q^5+1}\gamma^q \\
            &- y_2y_5a^{q^5+1}\gamma^{q^2} - y_2y_5a^{q^2+q}\gamma^{q^4} - y_2y_5a^{q^3+q^2}\gamma^q + 
            y_2y_5a^{q^3+q^2}\gamma^{q^2} - y_2y_5a^{q^3+q^2}\gamma^{q^3} + y_2y_5a^{q^3+q^2}\gamma^{q^4}\\
            &+ y_2y_5a^{q^5+q^2}\gamma^{q^3} - y_2y_5\gamma^{q^3+q} + y_2y_5\gamma^{q^4+q} + 
            y_2y_5\gamma^{q^3+q^2} - y_2y_5\gamma^{q^4+q^2} - y_2y_5\gamma^{2q^3} + y_2y_5\gamma^{q^4+q^3} \\
            &+ y_2a^{q^2+q+1}\gamma^{q^3} - y_2a^{q^2+q+1}\gamma^{q^4} + y_2a\gamma^{q^4+q} - 
            y_2a\gamma^{q^4+q^2} - y_2a^{q^2}\gamma^{q^3+q} + y_2a^{q^2}\gamma^{q^3+q^2} - y_2a^{q^2}\gamma^{2q^3}\\
            &+ y_2a^{q^2}\gamma^{q^4+q^3} + y_4y_5^2a^{q^5+q+1} - y_4y_5^2a^{2q^5+1} - 
            y_4y_5^2a^q\gamma^{q^2} - y_4y_5^2a^{q^5+q^4+q^3} + y_4y_5^2a^{q^3}\gamma^q + y_4y_5^2a^{2q^5+q^4}\\
            &- y_4y_5^2a^{q^5}\gamma^q + y_4y_5^2a^{q^5}\gamma^{q^2} + 
            y_4y_5a^{q^5+q^4+q+1} - y_4y_5a^{q+1}\gamma^{q^3} + y_4y_5a^{q+1}\gamma^{q^4} - y_4y_5a^{q^5+q^4+q^3+1}\displaybreak\\
            &+ y_4y_5a^{q^3+1}\gamma^q - y_4y_5a^{q^5+1}\gamma^{q^2} 
            + 2y_4y_5a^{q^5+1}\gamma^{q^3} - 2y_4y_5a^{q^5+1}\gamma^{q^4} - y_4y_5a^{q^4+q}\gamma^{q^2} + y_4y_5a^{q^4+q^3}\gamma^q \\
            &- y_4y_5a^{q^4+q^3}\gamma^{q^2} + y_4y_5a^{q^4+q^3}\gamma^{q^3} 
            - y_4y_5a^{q^4+q^3}\gamma^{q^4} - y_4y_5a^{q^5+q^4}\gamma^q + 2y_4y_5a^{q^5+q^4}\gamma^{q^2} - 2y_4y_5a^{q^5+q^4}\gamma^{q^3}\\
            &+ 2y_4y_5a^{q^5+q^4}\gamma^{q^4} - y_4y_5\gamma^{q^2+q} +
            y_4y_5\gamma^{q^3+q} - y_4y_5\gamma^{q^4+q} + y_4y_5\gamma^{2q^2} - y_4y_5\gamma^{q^3+q^2} + y_4y_5\gamma^{q^4+q^2} \\
            &-y_4a^{q^4+q+1}\gamma^{q^3} + y_4a^{q^4+q+1}\gamma^{q^4} + 
            y_4a^{q^4+q^3+1}\gamma^q - y_4a^{q^4+q^3+1}\gamma^{q^2} + y_4a^{q^4+q^3+1}\gamma^{q^3} - y_4a^{q^4+q^3+1}\gamma^{q^4} \\
            &+ y_4a\gamma^{q^3+q^2} - y_4a\gamma^{q^4+q^2} - y_4a\gamma^{2q^3}
            + 2y_4a\gamma^{q^4+q^3} - y_4a\gamma^{2q^4} - y_4a^{q^4}\gamma^{q^2+q} + y_4a^{q^4}\gamma^{q^3+q} - y_4a^{q^4}\gamma^{q^4+q} \\
            &+ y_4a^{q^4}\gamma^{2q^2} - 2y_4a^{q^4}\gamma^{q^3+q^2} + 
            2y_4a^{q^4}\gamma^{q^4+q^2} + y_4a^{q^4}\gamma^{2q^3} - 2y_4a^{q^4}\gamma^{q^4+q^3} + y_4a^{q^4}\gamma^{2q^4} + y_5^2a^{q^5+q^3+q+1} \\
            &- y_5^2a^{2q^5+q^3+1} - 
            y_5^2a^{q^5+q}\gamma^{q^2} + y_5^2a^{q^5+q^3}\gamma^{q^2} - y_5^2a^{q^5+q^3}\gamma^{q^3} + y_5^2a^{2q^5}\gamma^{q^3} + y_5a^{q^3+q+1}\gamma^{q^2} - y_5a^{q^3+q+1}\gamma^{q^3}\\
            &+ y_5a^{q^3+q+1}\gamma^{q^4} - y_5a^{q^5+q+1}\gamma^{q^2} + y_5a^{q^5+q+1}\gamma^{q^3} + y_5a^{q^5+q^3+1}\gamma^q - y_5a^{q^5+q^3+1}\gamma^{q^2} + y_5a^{q^5+q^3+1}\gamma^{q^3} \\
            &-  2y_5a^{q^5+q^3+1}\gamma^{q^4} - y_5a^q\gamma^{q^4+q^2} - y_5a^{q^3}\gamma^{q^2+q} + y_5a^{q^3}\gamma^{q^3+q} + y_5a^{q^3}\gamma^{2q^2} - 2y_5a^{q^3}\gamma^{q^3+q^2} + y_5a^{q^3}\gamma^{q^4+q^2} \\
            &+ y_5a^{q^3}\gamma^{2q^3} - y_5a^{q^3}\gamma^{q^4+q^3} - y_5a^{q^5}\gamma^{q^3+q} + 2y_5a^{q^5}\gamma^{q^3+q^2} - 2y_5a^{q^5}\gamma^{2q^3} + 2y_5a^{q^5}\gamma^{q^4+q^3} + a^{q+1}\gamma^{q^3+q^2}\\
            &- a^{q+1}\gamma^{q^4+q^2} - a^{q+1}\gamma^{2q^3} + a^{q+1}\gamma^{q^4+q^3} + a^{q^3+1}\gamma^{q^4+q} - a^{q^3+1}\gamma^{q^4+q^2} + a^{q^3+1}\gamma^{q^4+q^3} - a^{q^3+1}\gamma^{2q^4}\\
            &-\gamma^{q^3+q^2+q} + \gamma^{2q^3+q} -\gamma^{q^4+q^3+q} + \gamma^{q^3+2q^2} - 2\gamma^{2q^3+q^2} + 2\gamma^{q^4+q^3+q^2} + \gamma^{3q^3} - 2\gamma^{q^4+2q^3} + \gamma^{2q^4+q^3};
\end{align*}
\begin{align*}
r_3(\mathbf{y}):=&+y_2^2y_4y_5a^{q^2+q+1} - y_2^2y_4y_5a^{q^4+q+1} - y_2^2y_4y_5a^{2q^2+q} + y_2^2y_4y_5a^{q^4+q^2+q} + y_2^2y_4y_5a^{q^3+2q^2} \\
            &- 
            y_2^2y_4y_5a^{q^4+q^3+q^2} - y_2^2y_4y_5a^{q^2}\gamma + y_2^2y_4y_5a^{q^2}\gamma^q - y_2^2y_4y_5a^{q^2}\gamma^{q^2} + y_2^2y_4y_5a^{q^4}\gamma - 
            y_2^2y_4y_5a^{q^4}\gamma^q \\
            &+ y_2^2y_4y_5a^{q^4}\gamma^{q^2} + y_2^2y_4a^{q^3+2q^2+1}  - y_2^2y_4a^{q^4+q^3+q^2+1} + y_2^2y_4a^{q^2+1}\gamma^q - 
            y_2^2y_4a^{q^2+1}\gamma^{q^2} - y_2^2y_4a^{q^4+1}\gamma^q \\
            &+ y_2^2y_4a^{q^4+1}\gamma^{q^2} - y_2^2y_4a^{2q^2}\gamma + y_2^2y_4a^{q^4+q^2}\gamma + 
            y_2^2y_5a^{q^3+q^2+q+1} + y_2^2y_5a^{q^5+q^2+q+1} - y_2^2y_5a^{q+1}\gamma^{q^3} \\
            &- y_2^2y_5a^{q^5+q^3+q^2+1} - y_2^2y_5a^{q^5+1}\gamma^q + 
            y_2^2y_5a^{q^5+1}\gamma^{q^2} - y_2^2y_5a^{q^5+2q^2+q} - y_2^2y_5a^{q^2+q}\gamma + y_2^2y_5a^{q^2+q}\gamma^q\\
            & - y_2^2y_5a^{q^2+q}\gamma^{q^2} + 
            y_2^2y_5a^{q^2+q}\gamma^{q^3} + y_2^2y_5a^{q^5+q^3+2q^2} - y_2^2y_5a^{q^3+q^2}\gamma^q + y_2^2y_5a^{q^3+q^2}\gamma^{q^2} - y_2^2y_5a^{q^3+q^2}\gamma^{q^3} \\
            &+ 
            y_2^2y_5a^{q^5+q^2}\gamma^q - y_2^2y_5a^{q^5+q^2}\gamma^{q^2} + y_2^2y_5\gamma^{q+1} - y_2^2y_5\gamma^{q^2+1} + y_2^2y_5\gamma^{q^3+1} - y_2^2y_5\gamma^{2q} + 
            2y_2^2y_5\gamma^{q^2+q} \\
            &- y_2^2y_5\gamma^{q^3+q} - y_2^2y_5\gamma^{2q^2} + y_2^2y_5\gamma^{q^3+q^2} + y_2^2a^{q^3+2q^2+q+1} - y_2^2a^{q^2+q+1}\gamma^{q^3} + 
            y_2^2a^{q^2+q+1}\gamma^{q^4} - y_2^2a^{q^3+q^2+1}\gamma^{q^4} \\
            &- y_2^2a\gamma^{q^4+q} + y_2^2a\gamma^{q^4+q^2} - y_2^2a^{2q^2+q}\gamma + y_2^2a^{2q^2+q}\gamma^q - 
            y_2^2a^{2q^2+q}\gamma^{q^2} + y_2^2a^{2q^2+q}\gamma^{q^3} - y_2^2a^{2q^2+q}\gamma^{q^4} \\
            & - y_2^2a^{q^3+2q^2}\gamma^q + y_2^2a^{q^3+2q^2}\gamma^{q^2} - y_2^2a^{q^3+2q^2}\gamma^{q^3}
            + y_2^2a^{q^3+2q^2}\gamma^{q^4} + y_2^2a^{q^2}\gamma^{q+1} - y_2^2a^{q^2}\gamma^{q^2+1} + y_2^2a^{q^2}\gamma^{q^3+1} \\
            &- y_2^2a^{q^2}\gamma^{2q} + 2y_2^2a^{q^2}\gamma^{q^2+q} - 
            y_2^2a^{q^2}\gamma^{q^3+q} + y_2^2a^{q^2}\gamma^{q^4+q} - y_2^2a^{q^2}\gamma^{2q^2} + y_2^2a^{q^2}\gamma^{q^3+q^2} - y_2^2a^{q^2}\gamma^{q^4+q^2}\\
            & - y_2y_4y_5a^{q^4+q^3+q+1} - 
            y_2y_4y_5a^{q^5+q^4+q+1} + y_2y_4y_5a^{q+1}\gamma^{q^2} + y_2y_4y_5a^{q^5+q^3+q^2+1} + y_2y_4y_5a^{q^5+1}\gamma^q \\
            &- y_2y_4y_5a^{q^5+1}\gamma^{q^2} + 
            y_2y_4y_5a^{q^5+q^4+q^2+q} + y_2y_4y_5a^{q^2+q}\gamma - y_2y_4y_5a^{q^2+q}\gamma^q - y_2y_4y_5a^{q^2+q}\gamma^{q^2}\\
            & + y_2y_4y_5a^{q^4+q}\gamma^{q^2} - 
            y_2y_4y_5a^{q^5+q^4+q^3+q^2} - y_2y_4y_5a^{q^3+q^2}\gamma + 2y_2y_4y_5a^{q^3+q^2}\gamma^q - y_2y_4y_5a^{q^5+q^2}\gamma \\
            &+ y_2y_4y_5a^{q^4+q^3}\gamma - 
            y_2y_4y_5a^{q^4+q^3}\gamma^q + y_2y_4y_5a^{q^5+q^4}\gamma - y_2y_4y_5a^{q^5+q^4}\gamma^q + y_2y_4y_5a^{q^5+q^4}\gamma^{q^2} \\
            &- y_2y_4y_5\gamma^{q+1} + 
            y_2y_4y_5\gamma^{2q} - y_2y_4y_5\gamma^{q^2+q} - y_2y_4a^{q^4+q^3+q^2+q+1} + y_2y_4a^{q^4+q+1}\gamma^{q^3} - y_2y_4a^{q^4+q+1}\gamma^{q^4}\\
            & + 
            y_2y_4a^{q^3+q^2+1}\gamma^q + y_2y_4a^{q^3+q^2+1}\gamma^{q^2} - y_2y_4a^{q^3+q^2+1}\gamma^{q^3} +y_2y_4a^{q^3+q^2+1}\gamma^{q^4} - y_2y_4a^{q^4+q^3+1}\gamma^q \displaybreak\\
            &+ 
            y_2y_4a\gamma^{q^2+q} - y_2y_4a\gamma^{q^3+q} + y_2y_4a\gamma^{q^4+q} - y_2y_4a\gamma^{2q^2} + y_2y_4a\gamma^{q^3+q^2} - y_2y_4a\gamma^{q^4+q^2} + 
            y_2y_4a^{q^4+q^2+q}\gamma \\
            &- y_2y_4a^{q^4+q^2+q}\gamma^q + y_2y_4a^{q^4+q^2+q}\gamma^{q^2} - y_2y_4a^{q^4+q^2+q}\gamma^{q^3} + y_2y_4a^{q^4+q^2+q}\gamma^{q^4} + 
            y_2y_4a^{q^4+q^3+q^2}\gamma^q\\
            & - y_2y_4a^{q^4+q^3+q^2}\gamma^{q^2} + y_2y_4a^{q^4+q^3+q^2}\gamma^{q^3} - y_2y_4a^{q^4+q^3+q^2}\gamma^{q^4} - 2y_2y_4a^{q^2}\gamma^{q^2+1} + 
            y_2y_4a^{q^2}\gamma^{q^3+1} \\
            &- y_2y_4a^{q^2}\gamma^{q^4+1} - y_2y_4a^{q^4}\gamma^{q+1} + 2y_2y_4a^{q^4}\gamma^{q^2+1} - y_2y_4a^{q^4}\gamma^{q^3+1} + y_2y_4a^{q^4}\gamma^{q^4+1} + 
            y_2y_4a^{q^4}\gamma^{2q}\\
            & - 2y_2y_4a^{q^4}\gamma^{q^2+q} + y_2y_4a^{q^4}\gamma^{q^3+q} - y_2y_4a^{q^4}\gamma^{q^4+q} + y_2y_4a^{q^4}\gamma^{2q^2} - y_2y_4a^{q^4}\gamma^{q^3+q^2} + 
            y_2y_4a^{q^4}\gamma^{q^4+q^2} \\
            &+ y_2y_5a^{q^5+q^3+q^2+q+1} + y_2y_5a^{q^3+q+1}\gamma^{q^2} - y_2y_5a^{q^3+q+1}\gamma^{q^3} + y_2y_5a^{q^5+q+1}\gamma^{q^2} - 
            y_2y_5a^{q^5+q+1}\gamma^{q^3} \\
            &- y_2y_5a^{q^5+q^3+1}\gamma^q - 2y_2y_5a^{q^5+q^2+q}\gamma^{q^2} + y_2y_5a^{q^5+q^2+q}\gamma^{q^3} - y_2y_5a^q\gamma^{q^2+1} + 
            y_2y_5a^q\gamma^{q^2+q} - y_2y_5a^q\gamma^{2q^2} \\
            &+ y_2y_5a^q\gamma^{q^3+q^2} - y_2y_5a^{q^5+q^3+q^2}\gamma + y_2y_5a^{q^5+q^3+q^2}\gamma^q + 
            y_2y_5a^{q^5+q^3+q^2}\gamma^{q^2} - y_2y_5a^{q^5+q^3+q^2}\gamma^{q^3}\\
            & + y_2y_5a^{q^3}\gamma^{q+1} - y_2y_5a^{q^3}\gamma^{q^2+1} + y_2y_5a^{q^3}\gamma^{q^3+1} - y_2y_5a^{q^3}\gamma^{2q} +
            y_2y_5a^{q^3}\gamma^{q^2+q} - y_2y_5a^{q^3}\gamma^{q^3+q}\\
            & + y_2y_5a^{q^5}\gamma^{q^3+1} + y_2y_5a^{q^5}\gamma^{q^2+q} - y_2y_5a^{q^5}\gamma^{q^3+q} - y_2y_5a^{q^5}\gamma^{2q^2} + 
            y_2y_5a^{q^5}\gamma^{q^3+q^2} + 2y_2a^{q^3+q^2+q+1}\gamma^{q^2} \\
            &- 2y_2a^{q^3+q^2+q+1}\gamma^{q^3} + y_2a^{q^3+q^2+q+1}\gamma^{q^4} - y_2a^{q+1}\gamma^{q^3+q^2} + 
            y_2a^{q+1}\gamma^{q^4+q^2} + y_2a^{q+1}\gamma^{2q^3} - y_2a^{q+1}\gamma^{q^4+q^3}\\
            & - y_2a^{q^3+1}\gamma^{q^4+q} - 2y_2a^{q^2+q}\gamma^{q^2+1} + y_2a^{q^2+q}\gamma^{q^3+1} + 
            2y_2a^{q^2+q}\gamma^{q^2+q} - y_2a^{q^2+q}\gamma^{q^3+q} - 2y_2a^{q^2+q}\gamma^{2q^2}\\
            & + 3y_2a^{q^2+q}\gamma^{q^3+q^2} - 2y_2a^{q^2+q}\gamma^{q^4+q^2} - y_2a^{q^2+q}\gamma^{2q^3}
            + y_2a^{q^2+q}\gamma^{q^4+q^3} + y_2a^{q^3+q^2}\gamma^{q+1} - y_2a^{q^3+q^2}\gamma^{q^2+1} \\
            & + y_2a^{q^3+q^2}\gamma^{q^3+1} - y_2a^{q^3+q^2}\gamma^{q^4+1} - y_2a^{q^3+q^2}\gamma^{2q} + 
            y_2a^{q^3+q^2}\gamma^{q^4+q} + y_2a^{q^3+q^2}\gamma^{2q^2} - 2y_2a^{q^3+q^2}\gamma^{q^3+q^2} \\
            &+ y_2a^{q^3+q^2}\gamma^{q^4+q^2} + y_2a^{q^3+q^2}\gamma^{2q^3} - y_2a^{q^3+q^2}\gamma^{q^4+q^3} + 
            y_2\gamma^{q^2+q+1} - y_2\gamma^{q^3+q+1} - y_2\gamma^{2q^2+1}\\
            & + 2y_2\gamma^{q^3+q^2+1} - y_2\gamma^{2q^3+1} + y_2\gamma^{q^4+q^3+1} - y_2\gamma^{q^2+2q} + 
            y_2\gamma^{q^3+2q} + 2y_2\gamma^{2q^2+q} - 3y_2\gamma^{q^3+q^2+q}\\
            & + y_2\gamma^{q^4+q^2+q} + y_2\gamma^{2q^3+q} - y_2\gamma^{q^4+q^3+q} - y_2\gamma^{3q^2} + 
            2y_2\gamma^{q^3+2q^2} - y_2\gamma^{q^4+2q^2} - y_2\gamma^{2q^3+q^2} + y_2\gamma^{q^4+q^3+q^2}\\
            & - y_4y_5a^{q^5+q^4+q^3+q+1} + y_4y_5a^{q^5+q^3+1}\gamma^q + 
            y_4y_5a^{q^5+q^4+q}\gamma^{q^2} + y_4y_5a^q\gamma^{q^2+1} - y_4y_5a^q\gamma^{q^2+q} \\
            &+ y_4y_5a^{q^5+q^4+q^3}\gamma - y_4y_5a^{q^5+q^4+q^3}\gamma^q - 
            y_4y_5a^{q^3}\gamma^{q+1} + y_4y_5a^{q^3}\gamma^{2q} - y_4y_5a^{q^5}\gamma^{q^2+1} - y_4a^{q^4+q^3+q+1}\gamma^{q^2}\\
            & + y_4a^{q^4+q^3+q+1}\gamma^{q^3} - 
            y_4a^{q^4+q^3+q+1}\gamma^{q^4} + y_4a^{q^3+1}\gamma^{q^2+q} - y_4a^{q^3+1}\gamma^{q^3+q} + y_4a^{q^3+1}\gamma^{q^4+q} + y_4a^{q^4+q}\gamma^{q^2+1}\\
            & - y_4a^{q^4+q}\gamma^{q^2+q} + 
            y_4a^{q^4+q}\gamma^{2q^2} - y_4a^{q^4+q}\gamma^{q^3+q^2} + y_4a^{q^4+q}\gamma^{q^4+q^2} - y_4a^{q^4+q^3}\gamma^{q+1} + y_4a^{q^4+q^3}\gamma^{q^2+1} \\
            &- y_4a^{q^4+q^3}\gamma^{q^3+1} + 
            y_4a^{q^4+q^3}\gamma^{q^4+1} + y_4a^{q^4+q^3}\gamma^{2q} - y_4a^{q^4+q^3}\gamma^{q^2+q} + y_4a^{q^4+q^3}\gamma^{q^3+q} - y_4a^{q^4+q^3}\gamma^{q^4+q} \\
            &- y_4\gamma^{2q^2+1} + 
            y_4\gamma^{q^3+q^2+1} - y_4\gamma^{q^4+q^2+1} + y_5a^{q^5+q^3+q+1}\gamma^{q^2} - y_5a^{q^5+q^3+q+1}\gamma^{q^3} - y_5a^{q^5+q}\gamma^{2q^2}\\
            & + y_5a^{q^5+q}\gamma^{q^3+q^2} - 
            y_5a^{q^5+q^3}\gamma^{q^2+1} + y_5a^{q^5+q^3}\gamma^{q^3+1} + y_5a^{q^5+q^3}\gamma^{q^2+q} - y_5a^{q^5+q^3}\gamma^{q^3+q} + a^{q^3+q+1}\gamma^{2q^2} \\
            &- 2a^{q^3+q+1}\gamma^{q^3+q^2} + 
            a^{q^3+q+1}\gamma^{q^4+q^2} + a^{q^3+q+1}\gamma^{2q^3} - a^{q^3+q+1}\gamma^{q^4+q^3} - a^q\gamma^{2q^2+1} + a^q\gamma^{q^3+q^2+1}\\
            & + a^q\gamma^{2q^2+q} - a^q\gamma^{q^3+q^2+q} - 
            a^q\gamma^{3q^2} + 2a^q\gamma^{q^3+2q^2} - a^q\gamma^{q^4+2q^2} - a^q\gamma^{2q^3+q^2} + a^q\gamma^{q^4+q^3+q^2} + a^{q^3}\gamma^{q^2+q+1}\\
            & - a^{q^3}\gamma^{q^3+q+1} - a^{q^3}\gamma^{2q^2+1} + 
            2a^{q^3}\gamma^{q^3+q^2+1} - a^{q^3}\gamma^{q^4+q^2+1} - a^{q^3}\gamma^{2q^3+1} + a^{q^3}\gamma^{q^4+q^3+1} - a^{q^3}\gamma^{q^2+2q}\\
            & + a^{q^3}\gamma^{q^3+2q} + a^{q^3}\gamma^{2q^2+q} - 
            2a^{q^3}\gamma^{q^3+q^2+q} + a^{q^3}\gamma^{q^4+q^2+q} + a^{q^3}\gamma^{2q^3+q} - a^{q^3}\gamma^{q^4+q^3+q}.
\end{align*}

\end{document}